\documentclass[reqno,11pt]{amsart}
\usepackage[top=2.0cm,bottom=2.0cm,left=3cm,right=3cm]{geometry}
\usepackage{amsthm,amsmath,amssymb,dsfont}
\usepackage{mathrsfs,amsfonts,functan,extarrows,mathtools}

\usepackage{xcolor}
\usepackage{cite}
\usepackage{indentfirst, latexsym, amssymb, enumerate,amsmath,graphicx}
\usepackage{float}
\usepackage{relsize}
\usepackage{marginnote}
\usepackage{stmaryrd}
\usepackage{esint}
\usepackage{graphicx}
\usepackage{bm}

\usepackage{cite}
\allowdisplaybreaks[4]
\theoremstyle{plain}
\newtheorem{thm}{Theorem}[section]

\newtheorem{lem}[thm]{Lemma}
\newtheorem{prop}[thm]{Proposition}

\numberwithin{equation}{section}
\renewcommand{\theequation}{\thesection.\arabic{equation}}

\begin{document}

\title[The diffusion approximation model in
radiation hydrodynamics]
{Global existence and decay rates of strong solutions
to the diffusion approximation model in radiation hydrodynamics}

\author[P. Jiang]{Peng Jiang}
\address{School of Mathematics, Hohai University, Nanjing
 210098, P. R. China}
\email{syepmathjp@163.com}

\author[f.-C. Li]{fucai Li}
\address{Department of Mathematics, Nanjing University, Nanjing
 210093, P. R. China}
\email{fli@nju.edu.cn}

\author[J.-K. Ni]{Jinkai Ni$^*$} \thanks{$^*$\! Corresponding author}
\address{Department of Mathematics, Nanjing University, Nanjing
 210093, P. R. China}
\email{602023210006@smail.nju.edu.cn}

\begin{abstract}
In this paper, we study the global well-posedness and optimal time decay
rates of strong solutions to the diffusion approximation model
in radiation hydrodynamics in $\mathbb{R}^3$. This model consists of the full compressible 
Navier-Stokes equations and the radiative diffusion equation which describes
the influence and interaction between thermal radiation and fluid motion. 
Supposing that the initial perturbation around the equilibrium is sufficiently small in $H^2$-norm,
we  obtain the global strong solutions by utilizing method of the frequency decomposition. Moreover, by performing 
Fourier analysis techniques and using the delicate energy method, we consequently
derive the optimal decay rates (including highest-order derivatives)
of solutions for this model. 
\end{abstract}



\maketitle

\setcounter{equation}{0}
 \indent \allowdisplaybreaks

\section{Introduction}\label{Sec:intro-resul}
\subsection{Previous literature on our model}
In the study of radiation hydrodynamics, the propagation of thermal radiation and its effect on fluid motion have always been taken into account \cite{MM-1984-book,Pgc-2005,Cs-1960}.
As the temperature increases, the effect of thermal radiation becomes more pronounced. More precisely, physical research showed that the radiation intensity varies as the fourth power of the temperature.
Radiation hydrodynamics has a broad range of applications, such as supernova explosions, astrophysics, stellar evolution, and laser fusion (cf.\cite{KW-book-1994,PO-book-1968,BD-JMP-2006}). 

In this paper, we investigate the diffusion approximation (also called the Eddington approximation) model in radiation hydrodynamics, which reveals
energy flow arising from the radiative process in a semi-quantitative sense.
More precisely, when  the specific intensity of radiation
is almost isotropic \cite{Pgc-1973}, the governing equations of the diﬀusion approximation model for 3-D flow in Euler coordinates can be written   as follows (cf.\cite{Pgc-1973,BWWS-PT-1966}):
\begin{equation}\label{I-1}  
\left\{
\begin{aligned}
& \partial_t \varrho+{\rm div}(\varrho u)=0, \\
& \varrho (\partial_t u+ u\cdot\nabla u)+\nabla P=\mu \Delta u+(\lambda+\mu)\nabla {\rm div}u-\frac{1}{\mathfrak{C}}\nabla n, \\
& c_\nu\varrho (\partial_t \Theta+u\cdot \nabla\theta)=\kappa \Delta \Theta-P{\rm div}u+\lambda({\rm div}u)^2+2\mu D\cdot D-\Theta^4+n,\\
&\frac{1}{\mathfrak{C}}\partial_t n-\Delta n=\Theta^4-n, 
\end{aligned}\right.
\end{equation}
with the initial data
\begin{equation}\label{I--1}
(\varrho,u,\Theta,n)|_{t=0}=(\varrho_0(x),u_0(x),\Theta_0(x),n_0(x)), \quad x\in \mathbb{R}^{3}. 
\end{equation}
Here the unknown functions $\varrho=\varrho(x,t)>0$,\ $\Theta=\Theta(x,t)>0$,  $u=(u_1, u_2,u_3)=(u_1(x,t), \linebreak  u_2(x,t),u_3(x,t))\in\mathbb{R}^{3}$ and $n=n(t,x)\geq 0$
denote the density, temperature, velocity field of the fluid and radiation field respectively, with 
$x\in \mathbb{R}^{3}$ and $  t\geq 0$. Besides,
$P= R\varrho \Theta$ denotes the pressure. 
The parameters   $R$, $c_\nu$, and $\kappa$ are the physical constants, $\mathfrak{C}>0$ is light speed,
and $\lambda$ and $\mu>0$ are the viscosity coefficients 
 satisfying $3\lambda+2\mu>0$. And $D=D(u)$ is the deformation tensor
defined as
\begin{align*}
D_{ij}:=\frac{1}{2}\Big(\frac{\partial u_i }{\partial u_j} +\frac{\partial u_j }{\partial u_i}       \Big),      \quad   D\cdot D:=\sum_{i,j=3}^3 D_{ij}^2.
\end{align*}

Hereafter, we will recall some research results related to this model in the radiation hydrodynamics. 
Firstly, let's introduce the simplified model related to \eqref{I-1}. If relativistic effects are taken into account, the reciprocal of the light speed $\frac{1}{\mathfrak{C}}$ in $\eqref{I-1}_{4}$ can be taken as zero. Meanwhile, we assume that the eﬀect of radiation on the momentum could
be neglected. Then applying gradient to both sides of $\eqref{I-1}_{4}$ and denoting $q=-\nabla n$, we  deduce the following non-equilibrium or so-called radiating gas model (\!\cite{KNN,WX-2011-MMAS}):
\begin{equation}\label{NJKI-1}  
\left\{
\begin{aligned}
& \partial_t \varrho+{\rm div}(\varrho u)=0, \\
& \varrho (\partial_t u+ u\cdot\nabla u)+\nabla P=\mu \Delta u+(\lambda+\mu)\nabla {\rm div}u, \\
& C_\textsl{v}\varrho (\partial_t \Theta+u\cdot \nabla\theta)=\kappa \Delta \Theta-P{\rm div}u+\lambda({\rm div}u)^2+2\mu D\cdot D-{\rm div}q,\\
&-\nabla{\rm div}q+q+\nabla(\Theta^4)=0.
\end{aligned}\right.
\end{equation}
For 1-D case of the model \eqref{NJKI-1}, Kawashima-Nikkuni-Nishibata
\cite{K-N-N} firstly proved the global existence and the large time behavior of smooth solution near equilibrium state. In addition, there are some results for the stability of basic
waves. Coulombel et al. \cite{cgll, lcj} gave the existence and the stability of shock profile. Fan-Ruan-Xiang \cite{F-R-X1} investigated the stability of composite shock wave. The stability of a rarefaction wave was considered in \cite{Lcj-2011-cms, F-R-X2}. Furthermore, Wang-Xie \cite{WX-2011-nonlinear A}
proved the stability of viscous contact wave by elementary energy methods. 
Later, Hong \cite{Hh-arwa-2017} showed the stability of composite wave  under partially large initial perturbations. 
For multi-D case, Wang-Xie \cite{WX-2011-MMAS} studied the global existence
and large time behavior of smooth solutions for Cauchy problem when initial data is sufficient small.
For more results on \eqref{NJKI-1},  interested reader can
refer to \cite{WX-2011-JDE,Xf-DCDS-2012,RX-2013-MMAS} as well as the references therein.

Unfortunately, as it is clear 
from $\eqref{NJKI-1}_{2}$, the eﬀect of radiative pressure is excluded from the momentum equation,
which seems not very physically valid. Compared to the model \eqref{NJKI-1}, the diffusion approximation
model has an linear radiation effect term $\frac{1}{\mathfrak{C}}\nabla n$ in the momentum equation $\eqref{I-1}_{2}$.
However, since $\frac{1}{\mathfrak{C}}\nabla n$ lacks dissipative properties, it brings great difficulties in the study of global well-posedness. In fact, the  
classical energy method developed in \cite{MN-jmku-1980, K} doesn't work here. 
For 1-D case of model \eqref{I-1}, by ignoring the radiation effect term $n_{x}$ in momentum equation, Jiang-Xie-Zhang \cite{JXZ-2009} firstly proved the global existence of classical solution to the initial-boundary value problem with large initial data. Later, Jiang \cite{Jp-2015-DCDS} overcame the difficulties caused by the lower regularity of $\frac{1}{\mathfrak{C}}n_{x}$ and used refined $L^{p}$ $(p<2)$ estimates to obtain the global existence of classical solution for the full diﬀusion approximation model.
For 3-D case, also ignoring $\frac{1}{\mathfrak{C}}\nabla n$ in $\eqref{I-1}_{2}$, Jiang \cite{Jp-DCDS-2017} and Kim et al.
\cite{KHK-AMS-2023} studied the global existence and large time behavior of strong solutions for Cauchy problem 
around a constant state. Recently, Wang-Xie-Yang \cite{WXY-JDE-2023} used the Littlewood-Paley decomposition technique and frequency decomposition method combining with delicate energy estimates to obtain the global existence and time decay rates of classical solution (in $H^{4}(\mathbb{R}^3)$ space) of \eqref{I-1} in $\mathbb{R}^{3}$ with small initial data.
More research results on this model can be found in \cite{JZ-2018-JMAA, Jp-2018-ZAMP}.

\subsection{Our results}

In this paper, based on the results in \cite{WXY-JDE-2023}, we shall further study the global well-posedness and optimal time decay rates of strong solutions for the Cauchy problem \eqref{I-1}--\eqref{I--1} in the Sobolev space $H^{2}(\mathbb{R}^3)$.
Under the assumption that the initial perturbation around the equilibrium is
sufficiently small, we get the global existence of the strong solution.
Moreover, assuming that the initial data is bounded in $L^1$-norm, we obtain the
optimal time decay rates of all order derivatives (including highest-order derivatives) of this strong solution. Besides, by using
Sobolev’s interpolation inequality, we consequently derive the optimal time decay
rates in $L^{p}$-norm ($p\geq 2$).

Notice that $(\varrho, u, \Theta, n)\equiv(1,0,1,1)$ is an equilibrium
state of the system \eqref{I-1}, we shall set the standard perturbation $\varrho=1+\rho,\ \Theta=1+\theta,\  n=\eta+1$. Without loss of generality, we  take $\mu=\lambda=\kappa=R=C_\textsl{v}= \mathfrak{C} =1$. Then the Cauchy problem \eqref{I-1}-\eqref{I--1}
can be written as follows:
\begin{equation}\label{I-2}
\left\{
\begin{aligned}
& \partial_t \rho+(1+\rho){\rm div}u+\nabla \rho\cdot u=0, \\
& \partial_t u+u\cdot\nabla u+\frac{1+\theta}{1+\rho}\nabla\rho+\nabla\theta=\frac{\Delta u}{1+ \rho}+\frac{2\nabla {\rm div}u}{1+\rho}-\frac{\nabla \eta}{1+\rho}, \\
& \partial_t \theta+u\cdot\nabla\theta+(1+\theta){\rm div}u=\frac{\Delta \theta}{1+ \rho}+\frac{({\rm div}u)^2}{1+\rho}+\frac{2D\cdot D}{1+\rho}-\frac{(1+\theta)^4}{1+\rho}+\frac{1+\eta}{1+\rho},\\
& \partial_t \eta-\Delta\eta=(1+\theta)^4-(1+\eta),
\end{aligned}\right.
\end{equation}
with the initial data
\begin{align}\label{I--2}
(\rho,u,\theta,\eta)(x,t)|_{t=0}=\,&( \rho_{0}(x),u_{0}(x),\theta_{0}(x),\eta_0(x))\nonumber\\
:=\,&(\varrho_{0}(x)-1,u_{0}(x),\Theta_{0}(x)-1, n_0(x)-1 ), \quad x\in \mathbb{R}^{3}. .
\end{align}


The main results of this paper can be stated as follows.

\begin{thm}\label{T1.1}
Assume that $\|(\rho_0,u_0,\theta_0,\eta_0)\|_{H^2}$ is sufficiently small.
Then,
the Cauchy problem \eqref{I-2}--\eqref{I--2} admits a unique global 
strong solution $(\rho,u,\theta,\eta)$ satisfying  
\begin{gather*}
\rho,u,\theta,\eta \in C([0,\infty);H^{2}),
\end{gather*}
and
\begin{align*}
&\|(\rho,u,\theta,\eta)(\tau)\|_{H^{2}}^{2}+\int_{0}^{t}\Big(\|\nabla \rho(\tau)\|_{H^{1}}^{2}+\|\nabla(u,\theta,\eta)(\tau)\|_{H^{2}}^{2}+\|(4\theta-\eta)(\tau)\|_{H^2}^2\Big)\mathrm{d}\tau\nonumber\\
&\,\qquad \leq C\|(\rho,u,\theta,\eta)(0)\|_{H^{2}}^{2}.
\end{align*}
\end{thm}

\begin{thm}\label{T1.2}
Under the assumption of Theorem \ref{T1.1},  if we further assume that $\|(\rho_0,u_0,\theta_0,\linebreak \eta_0)\|_{L^1}$ is bounded. Then, it holds that  
\begin{align}
\|\nabla^m (\rho,u,\theta,\eta)(t)\|_{L^2}\leq&\, C(1+t)^{-\frac{3}{4}-\frac{m}{2}},
\end{align}  
for all $t\geq 0$ and $m=0,1,2$; and
\begin{align}\label{G1.9}
\|(\rho,u,\theta,\eta)(t)\|_{L^p}&\leq C(1+t)^{-\frac{3}{2}(1-\frac{1}{p})},
\end{align} 
for all $t\geq 0$ and $2\leq p\leq{\infty}$;
\begin{align}\label{G1.10}
\|\nabla(\rho,u,\theta,\eta)(t)\|_{L^p}&\leq C(1+t)^{-\frac{3}{2}(\frac{4}{3}-\frac{1}{p})},
\end{align} 
for all $t\geq 0$ and $2\leq p\leq 6$.
Moreover, 
\begin{align}
\|\partial_t (\rho,u)(t)\|_{L^2}\leq C(1+t)^{-\frac{5}{4}}, \\
\|\partial_t (\theta,\eta)(t)\|_{L^2}\leq C(1+t)^{-\frac{3}{4}}.
\end{align}
\end{thm}

\subsection{Our strategies to the proofs of Theorems \ref{T1.1} and \ref{T1.2} }
The main difficulty in proving Therorem \ref{T1.1} comes from the linear 
term $\nabla\eta$ in \eqref{I-2}$_2$. From \eqref{G3.51}, we know that the classical energy method
doesn't work here. Fortunately, we overcome this difficulty through dividing strong solution $(\rho, u,\theta,\eta)$ into low, medium and high frequency parts. 
More specifically, to begin with, we obtain the estimates of the highest-order derivative of strong solutions by utilizing the classical energy method (see Proposition \ref{P3.4}). Secondly, motivated by \cite{Wwj-siam-2021}, we develop a delicate energy method by making a change of variable (see \eqref{G3.52}--\eqref{G3.53}),
which is a critical step to deal with the estimate of solutions on the
zero-order derivative (see Proposition \ref{P3.7}). 
Similar to the method used in \cite{Wwj-siam-2021},
we establish two useful temporal energy functionals $\mathcal{H}(t)$ (see \eqref{G3.46}) and
$\mathcal{L}(t)$ (see \eqref{G3.68}) to get the estimate of $\|(\rho,u,\theta,\eta)\|_{H^2}$ by
making use of the equivalence condition $\mathcal{H}(t)+\mathcal{L}(t)\backsim
\|(\rho,u,\theta,\eta)\|_{H^2}^2$. 
Then, by combining Proposition \ref{P3.4} and Proposition \ref{P3.7}, we completed these estimates with
Lemma \ref{L3.8}, which is the key to prove the global existence of strong solutions.
Finally, based on the spectral analysis on the linear system \eqref{A-1}--\eqref{A--1} and Duhamel's principle, we derive the $L^2$-norm estimate
of medium-frequency part (see Lemma \ref{L3.9}--\ref{L3.10}). Substituting Lemma
\ref{L3.10} into Lemma \ref{L3.8}, we consequently get a priori estimates
and then establish the desired global well-posedness of strong solutions.

In Theorem \ref{T1.2}, we will establish the optimal time decay rates of all
derivatives of strong solutions $(\rho,u,\theta,\eta)$ in $L^2$-norm.
In particular, we can even get the time decay rate of the highest-order derivative  of solutions when
the standard energy method fails.
Additionally, according to Sobolev's interpolation inequality, we further obtain the optimal time decay
rates of strong solutions in $L^p$-norm $(p\geq 2)$.
Below, we describe our strategies to the proof of this theorem.
Firstly, by utilizing the temporal energy functional defined in \eqref{G3.46},
which is equivalent to $\|\nabla^2(\rho,u,\theta,\eta)\|_{L^2}$ (see \eqref{G3.50}). 
We   arrive at the inequality \eqref{G4.16} through Plancherel theorem, which is the key step in the   proof Theorem \ref{T1.2}.
Next, based on the spectral analysis and the frequency decomposition,
we establish the estimates of strong solutions
in $L^2$-norm in Lemmas \ref{L4.1}--\ref{L4.2}. With the above estimates in hand, 
we eventually complete the proof of Theorem \ref{T1.2}.

The remainder of this paper is stated as follows. In Section 2, we provide some
preliminaries containing notations and a series of lemmas.
In Section 3, we utilize the frequency decomposition and delicate energy method 
to obtain a priori estimates of solutions and then
establish the global existence of strong solutions to Cauchy problem \eqref{I-2}--\eqref{I--2}.
In Section 4, we adapt low-medium frequency part in appendix to 
prove the optimal time-decay rates of solutions. Finally, we give
the estimates on the linearized system corresponding to \eqref{I-2}--\eqref{I--2} in the Appendix.

\section{Preliminaries}

In this section, we shall introduce some important notations and some basic facts which will be used 
frequently. 

\subsection{Notations}
Throughout  this paper, $C>0$ denotes a generic constant. 
$A\sim B$ means that $\frac{A}{C}\leq B\leq CA$ for some positive constant $C$.
For all integer $m\geq 0$, $\nabla^m$ represents the m--th derivatives with
respect to the variable $x$.  We also denote $\partial_i=\partial_{x_i}$ for $i=1,2,3$, 
$\partial^{\alpha}_x=\partial^{\alpha_1}_{x_1}\partial^{\alpha_2}_{x_2}\partial^{\alpha_3}_{x_3}$ for some multi-indices $\alpha=(\alpha_1,\alpha_2,\alpha_3)$.
We  use $\langle\cdot, \cdot\rangle$ to represent the standard $L^2$ inner product in $\mathbb{R}^3$, i.e.
\begin{align*}
\langle f,g \rangle=\int_{\mathbb{R}^3} f(x)g(x) \mathrm{d}x,
\end{align*}
for any $f(x), g(x) \in L^2(\mathbb{R}^3)$.  $(f|g):=f\cdot g$ represents the dot product of $f$ 
with the complex conjugate of $g$.
If a function $h:\mathbb{R}^3\rightarrow \mathbb{R}$ is integrable, it has the following Fourier transform:
\begin{align*}
\mathcal{F}(h)(\xi)=\widehat{h}(\xi)=\int_{\mathbb{R}^{3}}e^{-ix\cdot\xi}h(x)dx,    
\end{align*}
where $i=\sqrt{-1}\in\mathbb{C}$ and $x\cdot\xi=\sum_{j=1}^{3}x_{j}\xi_{j}$ for 
any $\xi\in\mathbb{R}^{3}$.

The norm in $L^p$ 
space ($1\leq p \leq \infty$) is denoted by $\|\cdot\|_{L^p}$.
For  Sobolev space $H^k(\mathbb{R}^3)$ ($k\geq 0$), the corresponding norm is 
denoted by $\|\cdot\|_{H^k}$, i.e. 
\begin{align*}
\|h\|_{H^k}:=\sum_{|\alpha|\leq k}\|\partial^{\alpha} h\|_{L^2}.    
\end{align*}
 For brevity, we set $\|(a,b)\|_{X}:=\|a\|_{X}+\|b\|_{X}$ for Banach
space $X$, where $a=a(x)$ and $b=b(x)$ belong to $X$.

Next, we recall the frequency decomposition. Choose two smooth cut-off functions
$\phi_0(\xi)$ and $\phi_1(\xi)$ satisfying $0\leq \phi_0(\xi),\ \phi_1(\xi)\leq 1$
($\xi\in\mathbb{R}^3$), and
\begin{align*}
\phi_0(\xi)=\begin{cases}1, \quad|\xi|\leq \frac{r_0}{2},\\0, \quad|\xi|>{r_0},\end{cases}  \quad \phi_1(\xi)=\begin{cases}0, \quad|\xi|< \frac{R_0}{2},\\1, \quad|\xi|>{R_0}+1,\end{cases}
\end{align*}
where $r_0$ and  $R_0$ shall be determined later.
Then, we take $\phi_0(D_x)$ and $\phi_1(D_x)$ be the pseudo-differential
operators with $\phi_0(\xi)$ and $\phi_1(\xi)$ respectively.
For any $f(x)\in L^2(\mathbb{R}^3)$, we define
\begin{align}\label{G2.4}
f^l(x)=\phi_0(D_x)f(x),\quad f^m(x)=(1-\phi_0(D_x)-\phi_1(D_x))f(x), \quad
f^h(x)=\phi_1(D_x)f(x),
\end{align}
where $D_x:=\frac{1}{\sqrt{-1}}(\partial_{x_1},\partial_{x_2},\partial_{x_3})$.
It is obvious that 
\begin{align*}
f(x)=f^l(x)+f^m(x)+f^h(x).   
\end{align*}
We define $f^L(x):=f^l(x)+f^m(x)$ and $f^H(x):=f^m(x)+f^h(x)$.

Finally,  Let $\Lambda^s$ be the pseudo-differential operator 
defined as 
\begin{align*}
\Lambda^s g=\mathcal{F}^{-1}\big(|\xi|^s \mathcal{F}(g)\big),    
\end{align*}
for some $s\in \mathbb{R}$.

\subsection{Some useful lemmas}

With the above preparations in hand, we now state some useful facts which
will be used  frequently later.

\begin{lem}[see\cite{Jp-DCDS-2017,AF-Pa-2003}] \label{L2.1}
There exists a constant $C>0$ such that for any $f, g \in H^2\left(\mathbb{R}^3\right)$ and any multi-index $\alpha$ with $1 \leq|\alpha| \leq 2$,
\begin{align*}
\|f\|_{L^{\infty}\left(\mathbb{R}^3\right)} & \leq C\left\|\nabla f\right\|_{L^2\left(\mathbb{R}^3\right)}^{1 / 2}\left\|\nabla^2 f\right\|_{L^2\left(\mathbb{R}^3\right)}^{1 / 2}\leq C\| f\|_{H^2\left(\mathbb{R}^3\right)}, \\
\|f g\|_{H^1\left(\mathbb{R}^3\right)} & \leq C\|f\|_{H^1\left(\mathbb{R}^3\right)}\left\|\nabla g\right\|_{H^1\left(\mathbb{R}^3\right)}, \\
\left\|\partial^\alpha(f g)\right\|_{L^2\left(\mathbb{R}^3\right)} & \leq C\left\|\nabla f\right\|_{H^1\left(\mathbb{R}^3\right)}\left\|\nabla g\right\|_{H^1\left(\mathbb{R}^3\right)},\\
\|f \|_{L^r\left(\mathbb{R}^3\right)} & \leq C\|f\|_{H^1\left(\mathbb{R}^3\right)} \quad \text{for} \quad  2\leq r \leq 6. 
\end{align*}  
\end{lem}
\begin{lem}[see \cite{LZ-Cpaa-2020}]\label{L2.2}
Let $h$ and $g$ be two Schwarz functions. For $k\geq 0$, one has
\begin{align*}
\|\nabla^{k}(gh) \|_{L^r\left(\mathbb{R}^3\right)} & \leq C\|g\|_{L^{r_1}\left(\mathbb{R}^3\right)}\|\nabla^{k}h\|_{L^{r_2}
\left(\mathbb{R}^3\right)}+C\|h\|_{L^{r_3}\left(\mathbb{R}^3\right)}\|\nabla^{k}g
\|_{L^{r_4}\left(\mathbb{R}^3\right)},\\
\|\nabla^{k}(gh)-g\nabla^k h \|_{L^r\left(\mathbb{R}^3\right)} & \leq C\|\nabla g\|_{L^{r_1}\left(\mathbb{R}^3\right)}\|\nabla^{k-1}h\|_{L^{r_2}
\left(\mathbb{R}^3\right)}+C\|h\|_{L^{r_3}\left(\mathbb{R}^3\right)}\|\nabla^{k}g\|_{L^{r_4}
\left(\mathbb{R}^3\right)},    
\end{align*}
with $1<r,r_2,r_4<\infty$ and $r_i(1\leq i\leq 4)$ satisfy the following identity:
\begin{align*}
\frac{1}{r_1}+\frac{1}{r_2}=\frac{1}{r_3}+\frac{1}{r_4}=\frac{1}{r}.   
\end{align*}
\end{lem}
\begin{lem}[see\cite{Jp-DCDS-2017}]\label{L2.3}
Given any $0<\beta_1\neq1$
and $\beta_2>1$, it holds that
\begin{align*}
\int_0^t(1+t-s)^{-\beta_1}(1+s)^{-\beta_2}\mathrm{d}s\leq C(1+t)^{-\min\{\beta_1,\beta_2\}} 
\end{align*}
for all $t\geq 0$.
\end{lem}
\begin{lem}[see\cite{AF-Pa-2003}]\label{L2.4}
Suppose that $1\leq r\leq s\leq q\leq \infty$, and
\begin{align*}
\frac{1}{s}=\frac{\zeta}{r}+\frac{1-\zeta}{q},
\end{align*}
where $0\leq \zeta\leq 1$.
Let $f\in L^r(\mathbb{R}^3)\cap L^q(\mathbb{R}^3)$. Then $f\in L^s(\mathbb{R}^3)$, and
\begin{align*}
\|f\|_{L^s}\leq \|f\|_{L^r}^\zeta \|f\|_{L^q}^{1-\zeta}.    
\end{align*}
\end{lem}
\begin{lem}[see Appendix in\cite{Wwj-siam-2021}]\label{L2.5}
For any  $f \in H^m(\mathbb{R}^3)$,
it holds that
\begin{align*}
\|\nabla^k f^l\|_{L^2}\leq&\, r_0^{k-k_0}\|\nabla^{k_0 }f^l\|_{L^2},
\quad \|\nabla^k f^l\|_{L^2}\leq \|\nabla^k f\|_{L^2}, \nonumber\\
\|\nabla^k f^h\|_{L^2}\leq&\, \frac{1}{R_0^{k_1-k}}\|\nabla^{k_1 }f^h\|_{L^2},
\quad \|\nabla^k f^h\|_{L^2}\leq \|\nabla^k f\|_{L^2}, \nonumber\\
r_0^k\|f^m\|_{L^2}\leq&\, \|\nabla^k f^m\|_{L^2}\leq R_0^k\|f^m\|_{L^2},
\quad \|\nabla^k f^m\|_{L^2}\leq \|\nabla^k f\|_{L^2},
\end{align*}
where the integers $k,k_0,k_1$ and $m$ satisfy $k_0\leq k\leq k_1\leq m$.
\end{lem}

\section{Global existence of strong solutions }

In this section, we will prove the global existence of strong
solutions to the Cauchy problem \eqref{I-2}--\eqref{I--2}
in the whole space $\mathbb{R}^3$.

\subsection{A priori assumption}
To begin with, we assume that
\begin{align}\label{G3.1}
\sup_{0\leq t\leq T}\left\|(\rho,u,\theta,\eta)\right\|_{H^{2}}\leq\delta,
\end{align}
with a generic constant $0<\delta<1$  sufficiently small, and $(\rho,u,\theta,\eta)$
is strong solution to the problem  \eqref{I-2}--\eqref{I--2} on
$0\leq t\leq T$ for any $T>0$.

Notice that the problem  \eqref{I-2}--\eqref{I--2} can be 
expressed as
\begin{equation}\label{I-3}
\left\{
\begin{aligned}
& \partial_t \rho+{\rm div}u=\mathcal{N}_1, \\
& \partial_t u+\nabla\rho+\nabla\theta-\Delta u-2\nabla {\rm div}u+\nabla\eta=\mathcal{N}_2, \\
& \partial_t \theta+{{\rm div}u}-\Delta \theta+4\theta-\eta=\mathcal{N}_3,\\
& \partial_t \eta-\Delta\eta+\eta-4\theta=\mathcal{N}_4,
\end{aligned}\right.
\end{equation}
with the initial data
\begin{align*}
(\rho,u,\theta,\eta)(x,t)|_{t=0}=\,&( \rho_{0}(x),u_{0}(x),\theta_{0}(x),\eta_0(x))\nonumber\\
:=\,&(\varrho_{0}(x)-1,u_{0}(x),\Theta_{0}(x)-1, n_0(x)-1 ),
\end{align*}
where the nonlinear terms $\mathcal{N}_i(1\leq i\leq 4)$ are defined as
\begin{equation*}
\left\{
\begin{aligned}
\mathcal{N}_1:=&\,-\rho{\rm div}u-\nabla\rho\cdot u, \\
\mathcal{N}_2:=&\,-u\cdot\nabla u-\big(g(\rho)+h(\rho)\theta\big)\nabla\rho+g(\rho)\big(\Delta u+2\nabla{\rm div} u-\nabla\eta\big), \\
\mathcal{N}_3:=&\,g(\rho)\big(\Delta \theta+\eta-4\theta\big)-\theta{\rm div}u-u\cdot\nabla\theta+h(\rho)\big(({\rm div}u)^2+2D\cdot D-\theta^4-4\theta^3-6\theta^2\big),\\
\mathcal{N}_4:=&\,\theta^4+4\theta^3+6\theta^2, 
\end{aligned}\right.
\end{equation*}
and the nonlinear functions $g(\rho)$ and $h(\rho)$ are defined as  
\begin{equation*}
\left.
\begin{aligned}
g(\rho)=&\,\frac{1}{1+\rho}-1, \ \
h(\rho)=&\,\frac{1}{1+\rho}.
\end{aligned}\right.
\end{equation*}
Obviously, $g$ and $h$ own the following properties  
\begin{gather*}
|g(\rho)|\leq C\rho,\quad |h(\rho)|\leq C, \\
 |g^{(k)}(\rho)|,\   |h^{(k)}(\rho)|\leq C \ \  \mathrm{for\  any}\ \ k\geq1.    
\end{gather*}
Moreover, by Sobolev's inequality and \eqref{G3.1}, we obtain
\begin{align*}
    \frac{1}{2}\leq \rho+1\leq\frac{3}{2}.
\end{align*}

\medskip 
First, we pay attention to the estimates of strong solutions on the
highest order derivatives.

\subsection{Estimates of strong solutions on the
highest-order derivatives }
\begin{lem}
For the strong solutions to the problem \eqref{I-2}--\eqref{I--2}, there exist
a positive constant $\lambda_1$ such that
\begin{align}\label{G3.9}
&\frac{\mathrm{d}}{\mathrm{d}t}\|\nabla^{2}(\rho,u,\theta)(t)\|_{L^{2}}^{2}+\lambda_1\Big(\|\nabla^{3}u(t)\|_{L^{2}}^{2}+\|\nabla^{3}\theta(t)\|_{L^{2}}^{2}
+\|\nabla^{2}\mathrm{div}u(t)\|_{L^{2}}^{2}+\|\nabla^{2}\theta(t)\|_{L^{2}}^{2}\Big)\nonumber\\
&\,\quad \leq C\delta\|\nabla^{2}(\rho,u,\eta)(t)\|_{L^{2}}^{2}+\frac{1}{4}\Big(\|\nabla^3\eta(t)\|_{L^2}^2+\|\nabla^2\eta(t)\|_{L^2}^2 
    \Big)+C\|\nabla^{2}(u,\theta)(t)\|_{L^{2}}^{2}.   
\end{align}    
\end{lem}

\begin{proof}
It follows from  \eqref{I-3}$_2$--\eqref{I-3}$_4$ that 
\begin{align}\label{G3.aa}
&\frac{1}{2}\frac{\mathrm{d}}{\mathrm{d}t}\|\nabla^{2}(\rho,u,\theta)\|_{L^{2}}^{2}+\|\nabla^{3}( u,\theta)\|_{L^{2}}^{2}+2\|\nabla^{2}\mathrm{div}u\|_{L^{2}}^{2}+4\|\nabla^{2}\theta\|_{L^{2}}^{2}\nonumber\\
&\,\quad =-\langle\nabla^{2}{u},\nabla^{3}\eta\rangle+\langle\nabla^{2}{\theta},
\nabla^{2}\eta\rangle+\langle\nabla^{2}{\rho},\nabla^{2}\mathcal{N}_{1}\rangle
+\langle\nabla^{2}{u},\nabla^{2}\mathcal{N}_{2}\rangle+\langle\nabla^{2}{\theta},\nabla^{2}\mathcal{N}_{3}\rangle.    
\end{align}
Utilizing Young's inequality yields
\begin{align}\label{G3.11}
|\langle\nabla^{2}{u},\nabla^{3}\eta\rangle|\leq&\, \frac{1}{8}\|\nabla^3\eta\|_{L^2}^2+2\|\nabla^2 u\|_{L^2}^2,\\
|\langle\nabla^{2}{\eta},\nabla^{2}\theta\rangle|\leq&\, \frac{1}{8}\|\nabla^2\eta\|_{L^2}^2+2\|\nabla^2 \theta\|_{L^2}^2.\label{G3.11-1}
\end{align}
By the definition of $\mathcal{N}_1$, we get
\begin{align}\label{G3.12}
\langle\nabla^{2}{\rho},\nabla^{2}\mathcal{N}_{1}\rangle=-\langle \nabla^2\rho     ,\nabla^2(\rho{\rm div}u)\rangle-\langle \nabla^2\rho     ,\nabla^2(\nabla\rho\cdot u)\rangle.   
\end{align}
Thanks to Lemmas \ref{L2.1}--\ref{L2.2}, we obtain
\begin{align}
|\langle \nabla^2\rho     ,\nabla^2(\rho{\rm div}u)\rangle|\leq&\,  C\|\nabla^{2}{\rho}\|_{L^{2}}\Big(\|\nabla^{2}{\rho}\|_{L^{2}}\|\mathrm{div}{u}\|_{L^{\infty}}+\|{\rho}\|_{L^{\infty}}\|\nabla^{2}\mathrm{div}{u}\|_{L^{2}}\Big)\nonumber\\
\leq&\, C\|\nabla^{2}{\rho}\|_{L^{2}}\Big(\|\nabla^{2}{\rho}\|_{L^{2}}\|\nabla^2{u}\|_{H^{1}}+\|{\rho}\|_{H^{2}}\|\nabla^{3}{u}\|_{L^{2}}\Big)\nonumber\\
\leq&\, C\delta\big(\|\nabla^{2}{\rho}\|_{L^{2}}^{2}+\|\nabla^{3}{u}\|_{L^{2}}^{2}\big),\label{G3.13} \\
|\langle \nabla^2\rho     ,\nabla^2(\nabla\rho\cdot u)\rangle|\leq&\,
|\langle \nabla^2\rho     ,\nabla^2\nabla\rho\cdot u\rangle|+  
|\langle \nabla^2\rho     ,\nabla^2(\nabla\rho\cdot u)-\nabla^2\nabla\rho\cdot u\rangle|
\nonumber\\
\leq&\,  C\|\nabla^{2}{\rho}\|_{L^{2}}\Big(\|\nabla{\rho}\|_{L^{3}}\|\nabla^2{u}\|_{L^{6}}+\|{\nabla^2\rho}\|_{L^{2}}\|\nabla {u}\|_{L^{\infty}}\Big)\nonumber\\
\leq&\, C\|\nabla^{2}{\rho}\|_{L^{2}}\Big(\|{\rho}\|_{H^{2}}\|\nabla^3{u}\|_{L^{2}}+\|\nabla^2{\rho}\|_{L^{2}}\|\nabla^{2}{u}\|_{H^{1}}\Big)\nonumber\\
\leq&\, C\delta\big(\|\nabla^{2}{\rho}\|_{L^{2}}^{2}+\|\nabla^{3}{u}\|_{L^{2}}^{2}\big). \label{G3.14}
\end{align}
Here, we have used the fact:
\begin{align*}
\|\nabla^2{u}\|_{H^{1}}\|\nabla^{2}{\rho}\|_{L^{2}}^{2}\leq&\, C\delta\|\nabla^{2}{\rho}\|_{L^{2}}^{2}+C\delta\|\nabla^{3}{u}\|_{L^2}\|\nabla^{2}{\rho}\|_{L^{2}}\nonumber\\
\leq&\,
C\delta\big(\|\nabla^{2}{\rho}\|_{L^{2}}^{2}+\|\nabla^{3}{u}\|_{L^{2}}^{2}\big).   
\end{align*}
Substituting the above estimates 
 \eqref{G3.13}--\eqref{G3.14} 
into \eqref{G3.12},
we arrive at
\begin{align}\label{G3.16}
|\langle\nabla^2\rho,\nabla^2\mathcal{N}_1            \rangle|\leq
C\delta\big(\|\nabla^{2}{\rho}\|_{L^{2}}^{2}+\|\nabla^{3}{u}\|_{L^{2}}^{2}\big).  
\end{align}
By applying integration by parts and Young's inequality,  we deduce that
\begin{align}\label{G3.17}
|\langle\nabla^2{u},\nabla^2\mathcal{N}_2\rangle|\leq \frac{1}{4}\|\nabla^3{u}\|_{L^2}^2+\|\nabla\mathcal{N}_2\|_{L^2}^2.    
\end{align}
Notice that
\begin{align}\label{G3.18}
\|\nabla\mathcal{N}_{2}\|_{L^{2}}\leq\,& C\|\nabla({u}\cdot\nabla{u})\|_{L^{2}}+C\|\nabla(g({\rho})\nabla{\rho})\|_{L^{2}}+C\|\nabla(g({\rho})\Delta{u})\|_{L^{2}}\nonumber\\
&\,+C\|\nabla(g({\rho})\nabla{\rm div}{u})\|_{L^{2}}+C\|\nabla(g({\rho})\nabla\eta)\|_{L^{2}}+C\|\nabla(h({\rho})\theta\nabla\rho)\|_{L^{2}}. 
\end{align}
Then, the first two terms on the right hand-side of \eqref{G3.18} can be estimated as follows:
\begin{align*}\nonumber
\|\nabla({u}\cdot\nabla{u})\|_{L^{2}}\leq&\, C\|\nabla{u}\|_{L^{3}}\|\nabla{u}\|_{L^{6}}+C\|{u}\|_{L^{\infty}}\|\nabla^{2}{u}\|_{L^{2}}\nonumber\\
\leq&\, C\|{u}\|_{H^{2}}\|\nabla^{2}{u}\|_{L^{2}}\nonumber\\
\leq&\, C\delta\|\nabla^{2}{u}\|_{L^{2}}, \nonumber\\
\|\nabla(g({\rho})\nabla{\rho})\|_{L^{2}}\leq&\,  C\|\nabla g({\rho})\|_{L^{3}}\|\nabla{\rho}\|_{L^{6}}+C\|g({\rho})\|_{L^{\infty}}\|\nabla^{2}{\rho}\|_{L^{2}}\nonumber\\
\leq&\, C\|\nabla{\rho}\|_{L^{3}}\|\nabla{\rho}\|_{L^{6}}+C\|{\rho}\|_{L^{\infty}}\|\nabla^{2}{\rho}\|_{L^{2}}\nonumber\\
\leq&\, C\|{\rho}\|_{H^{2}}\|\nabla^{2}{\rho}\|_{L^{2}}\nonumber\\
\leq&\, C\delta\|\nabla^{2}{\rho}\|_{L^{2}}.
\end{align*}
Similarly, we have 
\begin{align*}\nonumber
\|\nabla(g({\rho})\Delta{u})\|_{L^{2}}+\|\nabla(g({\rho})\nabla{\rm div}{u})\|_{L^{2}}\leq&\, C\delta\|\nabla^3{u}\|_{L^2}, \nonumber\\  
\|\nabla(g({\rho})\nabla\eta)\|_{L^{2}}\leq&\, C\delta \|\nabla^2\eta\|_{L^2},
\end{align*}
and
\begin{align*}\nonumber
\|\nabla(h({\rho})\theta\nabla\rho)\|_{L^{2}}\leq&\, C\|\nabla h({\rho})\|_{L^{3}}\|\theta\nabla\rho\|_{L^{6}}+C\|h({\rho})\|_{L^{\infty}}\|\nabla(\theta\nabla\rho)\|_{L^{2}}\nonumber\\
\leq&\, C\|\nabla{\rho}\|_{L^{3}}\|\theta\|_{L^{\infty}}\|\nabla\rho\|_{L^6}+C\Big(\|\nabla\theta\|_{L^{6}}\|\nabla\rho\|_{L^{3}}+\|\theta\|_{L^{\infty}}\|\nabla^2\rho\|_{L^{2}}\Big)\nonumber\\
\leq&\, C\delta\big(\|\nabla^{2}{\rho}\|_{L^{2}}+\|\nabla^{2}\theta\|_{L^{2}}\big).    
\end{align*}
Plugging the above  estimates 
 into \eqref{G3.18} implies that 
\begin{align*}
\|\nabla\mathcal{N}_{2}\|_{L^{2}}\leq C\delta\big(\|\nabla^{2}({\rho},{u},\theta,\eta)\|_{L^{2}}+\|\nabla^{3}{u}\|_{L^{2}}\big),   
\end{align*}
which yields
\begin{align}\label{G3.23}
|\langle\nabla^2{u},\nabla^2\mathcal{N}_2\rangle|\leq C\delta\big(\|\nabla^2({\rho},{u},\theta,\eta)\|_{L^2}^2+\|\nabla^3{u}\|_{L^2}^2\big)+\frac{1}{4}\|\nabla^3 u\|_{L^2}^2.   
\end{align}
Similar to \eqref{G3.17}, we have
\begin{align}\label{G3.17-a}
|\langle\nabla^2\theta,\nabla^2\mathcal{N}_3\rangle|\leq \frac{1}{4}\|\nabla^3\theta\|_{L^2}^2+\|\nabla\mathcal{N}_3\|_{L^2}^2.        
\end{align}
Thus, we need to give the estimates of $\|\nabla \mathcal{N}_3\|_{L^2}$. We have 
\begin{align}\label{G3.25}
\|\nabla\mathcal{N}_{3}\|_{L^{2}}\leq\,& C\|\nabla({u}\cdot\nabla\theta)\|_{L^{2}}+C\|\nabla(\theta{\rm div}u)\|_{L^{2}}+C\|\nabla(g({\rho})\Delta{\theta})\|_{L^{2}}\nonumber\\
&\,+C\|\nabla(h({\rho})D\cdot D)\|_{L^{2}}+C\|\nabla(g({\rho})\eta)\|_{L^{2}}+C\|\nabla(g({\rho})\theta)\|_{L^{2}}\nonumber\\
&\,+C\|\nabla\big(h(\rho)(\theta^4+\theta^3+\theta^2)           \big)\|_{L^2}+C\|\nabla\big(h({\rho})({\rm div}{u})^2\big)\|_{L^{2}}. 
\end{align}
We start to handle the right hand-side of \eqref{G3.25} one by one.
Thanks to  Lemmas \ref{L2.1}--\ref{L2.2}, we have
\begin{align*}
\|\nabla(u\cdot\nabla\theta)\|_{L^{2}}\leq&\, C\|\nabla{\theta}\|_{L^{6}}\|\nabla{u}\|_{L^{3}}+C\|u\|_{L^{\infty}}\|\nabla^2\theta\|_{L^{2}}\nonumber\\
\leq&\, C\delta\|\nabla^{2}({u},\theta)\|_{L^{2}}, \nonumber\\
\|\nabla(\theta{\rm div}u)\|_{L^{2}}\leq&\, C\|\nabla{\theta}\|_{L^{6}}\|{\rm div}{u}\|_{L^{3}}+C\|\theta\|_{L^{\infty}}\|\nabla{\rm div}{u}\|_{L^{2}}\nonumber\\
\leq&\, C\delta\|\nabla^{2}({u},\theta)\|_{L^{2}}, \nonumber\\
\|\nabla(g({\rho})\Delta\theta)\|_{L^{2}}\leq&\,  C\|\nabla g({\rho})\|_{L^{3}}\|\nabla^2\theta\|_{L^{6}}+C\|g({\rho})\|_{L^{\infty}}\|\nabla^{3}\theta\|_{L^{2}}\nonumber\\
\leq&\, C\|\nabla{\rho}\|_{L^{3}}\|\nabla^2{\theta}\|_{L^{6}}+C\|{\rho}\|_{L^{\infty}}\|\nabla^{3}{\theta}\|_{L^{2}}\nonumber\\
\leq&\, C\delta\|\nabla^{3}{\theta}\|_{L^{2}},\nonumber\\
\|\nabla\big(h({\rho})({\rm div}{u})^2\big)\|_{L^{2}}+\|\nabla(h({\rho})D\cdot D)\|_{L^{2}}\leq&\,C\|\nabla h(\rho)\|_{L^6}\|\nabla u\|_{L^6}^2+C\|\nabla^2 u\|_{L^6}\|\nabla u\|_{L^3}\nonumber\\
\leq&\,C\|\nabla\rho\|_{L^6}\|\nabla^2 u\|_{L^2}^2+C\|u\|_{H^2}\|\nabla^3 u\|_{L^2}\nonumber\\
\leq&\,C\delta\big(\|\nabla^2 u\|_{L^2}+\|\nabla^3 u\|_{L^2}           \big),\nonumber\\
\|\nabla(g({\rho})\theta)\|_{L^{2}}+\|\nabla(g({\rho})\eta)\|_{L^{2}}\leq&\,  C\|\nabla g({\rho})\|_{L^{6}}\|(\theta,\eta)\|_{L^{3}}+C\|g({\rho})\|_{L^{3}}\|\nabla(\theta,\eta)\|_{L^{6}}\nonumber\\
\leq&\, C\|\nabla{\rho}\|_{L^{6}}\|(\theta,\eta)\|_{L^{3}}+C\|{\rho}\|_{L^{3}}\|\nabla({\theta},\eta)\|_{L^{6}}\nonumber\\
\leq&\, C\|\nabla^2{\rho}\|_{L^{2}}\|(\theta,\eta)\|_{H^{1}}+C\|{\rho}\|_{H^{1}}\|\nabla^2({\theta},\eta)\|_{L^{2}}\nonumber\\
\leq&\, C\delta\big(\|\nabla^{2}\rho\|_{L^{2}}+\|\nabla^{2}\eta\|_{L^{2}}+\|\nabla^{2}\theta\|_{L^{2}}\big),\nonumber\\
\|\nabla\big(h(\rho)(\theta^4+\theta^3+\theta^2)     \big)\|_{L^2}\leq&\,
C\|\nabla h(\rho)\|_{L^6}\|\theta\|_{L^3}\Big( \|\theta\|_{L^{\infty}}+\|\theta\|_{L^{\infty}}^2+\|\theta\|_{L^{\infty}}^3             \Big)\nonumber\\
&\,+C\|h(\rho)\|_{\infty}\|\nabla(\theta^4+\theta^3+\theta^2)\|_{L^2}\nonumber\\
\leq &\,C\delta\|\nabla^2\rho\|_{L^2}+C\delta\|\nabla\theta\|_{L^6}\Big( 1+\|\theta\|_{L^{\infty}}+\|\theta\|_{L^{\infty}}^2             \Big)\nonumber\\
\leq&\,C\delta\big(\|\nabla^{2}\rho\|_{L^{2}}+\|\nabla^{2}\theta\|_{L^{2}}\big).
\end{align*}
Putting the above estimates into \eqref{G3.25}, and combining it with \eqref{G3.17-a}, we directly get
\begin{align}\label{G3.27}
|\langle\nabla^2{\theta},\nabla^2\mathcal{N}_3\rangle|\leq C\delta\big(\|\nabla^2({\rho},{u},\theta,\eta)\|_{L^2}^2+\|\nabla^3({u},\theta)\|_{L^2}^2\big)+\frac{1}{4}\|\nabla^3\theta\|_{L^2}^2.   
\end{align}
Therefore, by putting \eqref{G3.11}, \eqref{G3.16},  \eqref{G3.23}
and \eqref{G3.27} to \eqref{G3.aa}, we consequently arrive at the desired inequality \eqref{G3.9}.
\end{proof}

\begin{lem}
For the strong solutions to the problem \eqref{I-2}--\eqref{I--2}, it holds that
\begin{align}\label{G3.28}
&\frac{\mathrm{d}}{\mathrm{d}t}\|\nabla^{2}\eta(t)\|_{L^{2}}^{2}+\|\nabla^{3}\eta(t)\|_{L^{2}}^{2}+\|\nabla^{2}\eta(t)\|_{L^{2}}^{2}\leq C\delta\|\nabla^{2}\theta(t)\|_{L^{2}}^{2}+C\|\nabla^{2}\theta(t)\|_{L^{2}}^{2}.   
\end{align}    
\end{lem}

\begin{proof}
It follows from \eqref{I-3}$_4$ that
\begin{align}\label{G3.29}
&\frac{1}{2}\frac{\mathrm{d}}{\mathrm{d}t}\|\nabla^{2}\eta\|_{L^{2}}^{2}+\|\nabla^{3}\eta\|_{L^{2}}^{2}+\|\nabla^{2}\eta\|_{L^{2}}^{2} =\langle \nabla^2\eta,\nabla^2\mathcal{N}_4\rangle+4\langle \nabla^2\theta,\nabla^2\eta\rangle.
\end{align}
By using integration by parts and  Young's inequality, we have
\begin{align} 
4\langle \nabla^2\theta,\nabla^2\eta\rangle\leq&\, \frac{1}{2}\|\nabla^2\eta\|_{L^2}^2+8\|\nabla^2\theta\|_{L^2}^2, \label{G3.30}\\ 
\langle \nabla^2\eta,\nabla^2\mathcal{N}_4\rangle\leq&\, \frac{1}{2}\|\nabla^3\eta\|_{L^2}^2+2\|\nabla\mathcal{N}_4\|_{L^2}^2.\label{G3.30-1}
\end{align}
Similar to \eqref{G3.25},
\begin{align}\label{G3.31}
\|\nabla\mathcal{N}_4\|_{L^2}\leq&\, C\|\theta\|_{L^3}\|\nabla\theta\|_{L^6}\Big(1+\|\theta\|_{L^{\infty}}+\|\theta\|_{L^{\infty}}^2\Big)\nonumber\\
\leq&\, C\delta\|\nabla^2\theta\|_{L^2}.
\end{align}
Inserting \eqref{G3.30}--\eqref{G3.31} into \eqref{G3.29},
we get \eqref{G3.28}.    
\end{proof}

Combining  \eqref{G3.9} with \eqref{G3.28} gives
\begin{align}\label{G3.32}
&\frac{\mathrm{d}}{\mathrm{d}t}\|\nabla^{2}(\rho,u,\theta,\eta)(t)\|_{L^{2}}^{2}+\lambda_2\Big(\|\nabla^{3}(u,\theta,\eta)(t)\|_{L^{2}}^{2}
+\|\nabla^{2}\mathrm{div}u(t)\|_{L^{2}}^{2}+\|\nabla^{2}(\theta,\eta)(t)\|_{L^{2}}^{2}\Big)\nonumber\\
&\quad \,\leq C\delta\|\nabla^{2}(\rho,u)(t)\|_{L^{2}}^{2}+C\|\nabla^{2}(u,\theta)(t)\|_{L^{2}}^{2},
\end{align}    
for some $0<\lambda_2<1$.

\medskip
Next, we give the estimate of $\nabla^2 \rho$.
\begin{lem}
For the strong solutions to the problem \eqref{I-2}--\eqref{I--2}, it holds that    
\begin{align}\label{G3.33}
&-\frac{\mathrm{d}}{\mathrm{d}t}\langle\nabla{\rm div}{u}(t),\nabla{\rho}^{h}(t)\rangle+\frac{1}{4}\|\nabla^{2}{\rho}(t)\|_{L^{2}}^{2}\nonumber\\
&\,\quad \leq C\Big(\|\nabla^{2}{\rho}^{L}(t)\|_{L^{2}}^{2}+\|\nabla^2(\theta^h,\eta^h)(t)\|_{L^2}^2\Big)\\
&\,\quad \ \ \ +4\Big( \|\nabla^{2}{u}(t)\|_{L^{2}}^{2}+\|\nabla^{3}{u}(t)\|_{L^{2}}^{2}+\|\nabla^{2}\mathrm{div}u(t)\|_{L^{2}}^{2}    \Big)\nonumber\\
&\,\quad\ \ \ \,+C\delta\Big(\|\nabla^2(u,\theta,\eta)\|_{L^2}^2+\|\nabla^3 u\|_{L^2}^2        \Big).   
\end{align}
\end{lem}
\begin{proof}
From \eqref{I-3}$_2$, we have
\begin{align}\label{G3.34}
\nabla{\rho}=-\partial_{t}{u}+\Delta{u}+2\nabla{\rm div}{u}+\mathcal{N}_2-\nabla\theta-\nabla\eta.   
\end{align}
Applying $\nabla$ to \eqref{G3.34} and multiplying by
the result $\nabla^2\rho^h$, then using integration by parts, we obtain
\begin{align}\label{G3.35}
&-\frac{\mathrm{d}}{\mathrm{d}t}\langle\nabla{\rm div}{u},\nabla{\rho}^{h}\rangle+\langle\nabla^2{\rho}^{h},\nabla^2{\rho}\rangle\nonumber\\
&\,\quad =-\langle\nabla{\rm div}{u},\partial_{t}\nabla{\rho}^{h}\rangle+\langle\nabla^2{\rho}^{h},\nabla\Delta{u}\rangle+2\langle\nabla^2{\rho}^{h},\nabla^2{\rm div}{u}\rangle\nonumber\\
&\,\quad \ \ \ \, -\langle\nabla^2{\rho}^{h},\nabla^2\theta\rangle-\langle\nabla^2{\rho}^{h},\nabla^2\eta\rangle+\langle\nabla^2{\rho}^{h},\nabla\mathcal{N}_{2}\rangle.    
\end{align}
According to Young's inequality, it yields
\begin{align}\label{G3.36}
\langle\nabla^2{\rho}^{h},\nabla^2{\rho}\rangle=&\,\|\nabla^2{\rho}\|_{L^{2}}^{2}-\langle\nabla^2{\rho}^{L},\nabla^2{\rho}\rangle\nonumber\\
\geq&\,\frac{7}{8}\|\nabla^{2}{\rho}\|_{L^{2}}^{2}-2\|\nabla^{2}{\rho}^{L}\|_{L^{2}}^{2}.   
\end{align}
Applying the operator $\phi_1(D_x)$ to \eqref{I-3}$_1$, we derive that
\begin{align}\label{G3.37}
\partial_t{\rho}^h=-{\rm div}{u}^h+\mathcal{N}_1^h.   
\end{align}
Applying $\nabla$ to \eqref{G3.37} and multiplying 
the result 
by $\nabla{\rm div}u$ gives
\begin{align*}
-\langle\nabla{\rm div}{u},\partial_{t}\nabla{\rho}^{h}\rangle=\langle\nabla{\rm div}{u},\nabla{\rm div}{u}^{h}\rangle-\langle\nabla{\rm div}{u},\nabla\mathcal{N}_{1}^{h}\rangle. 
\end{align*}
Then, by making use of Young's inequality, we get
\begin{align}\label{G3.39}
-\langle\nabla{\rm div}{u},\partial_{t}\nabla{\rho}^{h}\rangle\leq&\,\frac{3}{2}\|\nabla{\rm div}{u}\|_{L^{2}}^{2}+\frac{1}{2}\|\nabla{\rm div}{u}^{h}\|_{L^{2}}^{2}+\frac{1}{4}\|\nabla\mathcal{N}_{1}^{h}\|_{L^{2}}^{2}\nonumber\\
\leq&\,2\|\nabla{\rm div}{u}\|_{L^{2}}^{2}+\frac{1}{4}\|\nabla\mathcal{N}_{1}\|_{L^{2}}^{2}\nonumber\\
\leq&\,2\|\nabla^2{u}\|_{L^{2}}^{2}+C\delta\|\nabla^{2}({\rho},{u})\|_{L^{2}}^{2},    
\end{align}
Here, we have used the fact that 
\begin{align}\label{G3.40}
\frac{1}{4}\|\nabla\mathcal{N}_{1}\|_{L^{2}}\leq& \,C\|{\rho}\|_{L^{\infty}}\|\nabla^{2}{u}\|_{L^{2}}+C\|\nabla^{2}{\rho}\|_{L^{2}}\|{u}\|_{L^{\infty}}+C\|\nabla\rho\|_{L^6}\|\nabla u\|_{L^3}\nonumber\\
\leq&\, C\delta\|\nabla^{2}({\rho},{u})\|_{L^{2}}.  
\end{align}
By an argument similar to \eqref{G3.39}, we infer that
\begin{align}\label{G3.41-1}
&\langle\nabla^2{\rho}^{h},\nabla\Delta{u}\rangle+2\langle\nabla^2{\rho}^{h},\nabla^2{\rm div}{u}\rangle\nonumber\\
\leq&\,\frac{3}{8}\|\nabla^2{\rho}^{h}\|_{L^{2}}^{2}+2\|\nabla\Delta{u}\|_{L^{2}}^{2}+4\|\nabla^2{\rm div}{u}\|_{L^{2}}^{2}\nonumber\\
\leq&\,\frac{3}{8}\|\nabla^2{\rho}\|_{L^{2}}^{2}
+2\|\nabla^{3}{u}\|_{L^{2}}^{2}+4\|\nabla^{2}\mathrm{div}u\|_{L^{2}}^{2}.   
\end{align}
Then, by using the definition of \eqref{G2.4} 
and Plancherel theorem, we arrive at
\begin{align*}
-\langle\nabla^2{\rho}^{h},\nabla^2\theta\rangle=&\,-\langle\phi_{1}(D_{x})\nabla^2{\rho},\nabla^2\theta\rangle\nonumber\\
=&\,-\langle\nabla^2{\rho},\phi_{1}(D_{x})\nabla^2\theta\rangle\nonumber\\
\leq&\,\frac{1}{16}\|\nabla^2{\rho}\|_{L^{2}}^{2}+4\|\nabla^2\theta^{h}\|_{L^{2}}^{2},   
\end{align*}
and
\begin{align*}
-\langle\nabla^2{\rho}^{h},\nabla^2\eta\rangle
\leq&\,\frac{1}{16}\|\nabla^2{\rho}\|_{L^{2}}^{2}+4\|\nabla^2\eta^{h}\|_{L^{2}}^{2}.   
\end{align*}
For the remaining terms on the right hand-side   of \eqref{G3.35}, it
follows from Young's inequality that
\begin{align}\label{G3.44}
\langle\nabla^2{\rho}^{h},\nabla\mathcal{N}_{2}\rangle\leq&\,\frac{1}{8}\|\nabla^2{\rho}^{h}\|_{L^{2}}^{2}+2\|\nabla\mathcal{N}_{2}\|_{L^{2}}^{2}\nonumber\\
\leq&\,\frac{1}{8}\|\nabla^2{\rho}\|_{L^{2}}^{2}+C\delta\Big(\|\nabla^{2}(\rho,u,\theta,\eta)\|_{L^{2}}^{2}+\|\nabla^{3}{u}\|_{L^{2}}^{2}\Big).  
\end{align}
By plugging \eqref{G3.36}, \eqref{G3.39}, \eqref{G3.41-1}--\eqref{G3.44} into \eqref{G3.35},
we consequently get \eqref{G3.33}.
\end{proof}
Adding \eqref{G3.32} and $\frac{1}{8}\lambda_2\times \eqref{G3.33}$, we get
\begin{align}\label{G3.45}
&\frac{\mathrm{d}}{\mathrm{d}t}\Big(\|\nabla^{2}(\rho,u,\theta,\eta)(t)\|_{L^{2}}^{2}-\frac{1}{8}\lambda_2\langle\nabla{\rm div}{u}(t),\nabla{\rho}^{h}(t)\rangle\Big)\nonumber\\
&\,+\lambda_3\Big(\|\nabla^{3}(u,\theta,\eta)(t)\|_{L^{2}}^{2}+\|\nabla^{2}(\theta,\eta,\rho)(t)\|_{L^{2}}^{2}+\|\nabla^{2}\mathrm{div}u(t)\|_{L^{2}}^{2}  \Big)\nonumber\\
&\,\quad \leq C_1\Big(\|\nabla^{2}(u,\theta)(t)\|_{L^{2}}^{2}+\|\nabla^2\rho^L(t)\|_{L^2}^2+\|\nabla^2\eta^h(t)\|_{L^2}^2\Big),
\end{align}  
for some  $\lambda_3>0$. Here, the positive constant $C_1$ is independent of 
$r_0$, $R_0$ and $T$.
Now, we define the temporal energy functional
\begin{align}\label{G3.46}
\mathcal{H}(t):=\|\nabla^{2}(\rho,u,\theta,\eta)(t)\|_{L^{2}}^{2}-\frac{1}{8}\lambda_2\langle\nabla{\rm div}{u}(t),\nabla{\rho}^{h}(t)\rangle.  
\end{align}
Utilizing Young's inequality and Lemma \ref{L2.5}, we achieve that
\begin{align*}
-\frac{1}{8}\lambda_2\langle\nabla{\rm div}{u}(t),\nabla{\rho}^{h}(t)\rangle
\leq&\, \frac{1}{16}\lambda_2\|\nabla^2 u\|_{L^2}^2+ \frac{1}{16}\lambda_2\|\nabla\rho^h\|_{L^2}^2\nonumber\\
\leq&\,\frac{1}{16}\|\nabla^2 u\|_{L^2}^2+\frac{1}{16R_0^2}\|\nabla^2 \rho\|_{L^2}^2\nonumber\\
\leq&\,\frac{1}{16}\|\nabla^2 u\|_{L^2}^2+\frac{1}{16}\|\nabla^2 \rho\|_{L^2}^2,
\end{align*}
where we have used the facts that $0<\lambda_2<1$ and $R_0>1$.
As a result,
there exists a positive constant $C_{2}$, which is 
independent of $\delta$ such that
\begin{align}\label{G3.48}
\frac{1}{C_{2}}\|\nabla^{2}({\rho},{u},\theta,\eta)(t)\|_{L^{2}}^{2}\leq \mathcal{H}(t)\leq C_{2}\|\nabla^{2}({\rho},{u},\theta,\eta)(t)\|_{L^{2}}^{2}.    
\end{align}

With the help of \eqref{G3.45}, \eqref{G3.46} and \eqref{G3.48},
we give the   following proposition.

\begin{prop}\label{P3.4}
For the strong solutions to the problem \eqref{I-2}--\eqref{I--2}, it holds that
\begin{align}\label{G3.49}
&\frac{\mathrm{d}}{\mathrm{d}t}\mathcal{H}(t)+\lambda_3\Big(\|\nabla^{3}(u,\theta,\eta)(t)\|_{L^{2}}^{2}+\|\nabla^{2}(\theta,\eta,\rho)(t)\|_{L^{2}}^{2}+\|\nabla^{2}\mathrm{div}u(t)\|_{L^{2}}^{2}  \Big)\nonumber\\
&\,\quad \leq C_1\Big(\|\nabla^{2}(u,\theta)(t)\|_{L^{2}}^{2}+\|\nabla^2\rho^L(t)\|_{L^2}^2+\|\nabla^2\eta^h(t)\|_{L^2}^2\Big),
\end{align}
with
\begin{align}\label{G3.50}
\mathcal{H}(t)\backsim    \|\nabla^{2}({\rho},{u},\theta,\eta)(t)\|_{L^{2}}^{2}. 
\end{align}    
\end{prop}

\medskip 
Next, we   devote ourselves to the estimates of strong solutions on
the zero derivative.

\subsection{Estimates of solutions on the zero-order derivative}
What we want to emphasize is that the classical energy method doesn't work here.
In fact, by utilizing integration by parts and Young's inequality, we get
\begin{align}\label{G3.51}
&\frac{1}{2}\frac{\mathrm{d}}{\mathrm{d}t}\int_{\mathbb{R}^{3}}\Big(|{2\rho}|^{2}+|2{u}|^{2}+|2\theta|^2+|\eta|^2\Big)\mathrm{d}x\nonumber\\
&\,+\int_{\mathbb{R}^{3}}\Big(4|\nabla{u}|^{2}+4|\nabla{\theta}|^{2}+|\nabla\eta|^2+8|{\rm div}{u}|^{2}+|4\theta-\eta|^2\Big)\mathrm{d}x\nonumber\\
=&\,-\int_{\mathbb{R}^{3}}{u}\cdot\nabla\eta\mathrm{d}x+\int_{\mathbb{R}^{3}}\Big({\rho}\mathcal{N}_{1}+{u}\cdot\mathcal{N}_{2}+\theta\mathcal{N}_{3}+\eta\mathcal{N}_{4}\Big)\mathrm{d}x\nonumber\\
\leq&\,\frac{1}{2}\int_{\mathbb{R}^{3}}|u|^{2}\mathrm{d}x+\frac{1}{2}\int_{\mathbb{R}^{3}}|\nabla\eta|^{2}\mathrm{d}x+\int_{\mathbb{R}^{3}}\Big({\rho}\mathcal{N}_{1}
+{u}\cdot\mathcal{N}_{2}+\theta\mathcal{N}_{3}+\eta\mathcal{N}_{4}\Big)\mathrm{d}x.
\end{align}
Because of the fact that the estimate of $\int_{\mathbb{R}^{3}}|u|^{2}\mathrm{d}x$ can't be controlled
directly, we shall develop a delicate  method in the same spirit of  \cite{Wwj-siam-2021} to deal with our system.
Let's consider a change of variable as follows:
\begin{align}\label{G3.52} \left(
\begin{matrix}
\rho  \\
u   \\
G  \\
F  \\
\end{matrix}\right):=\mathbb{T}\left(
\begin{matrix}
{\rho}  \\
{u}   \\
{\theta}  \\
\eta \\
\end{matrix}\right),
\end{align}
where
\begin{align}\label{G3.53}
\mathbb{T}=\left(
\begin{matrix}
1 & 0 & 0 & 0   \\
0 & 1 & 0 & 0    \\
0 & 0 & 4 & -1   \\
0 & 0 & 1 & 1  \\
\end{matrix}\right).
\end{align}
Therefore, the system \eqref{I-3} could be represented as
\begin{equation}\label{G3.54}
\left\{
\begin{aligned}
& \partial_t \rho+{\rm div}u=\mathcal{N}_1, \\
& \partial_t u+\nabla\rho-\Delta u-2\nabla {\rm div}u+\nabla F=\mathcal{N}_2, \\
& \partial_t G+4{\rm div}u-\Delta G+5G=4\mathcal{N}_3-\mathcal{N}_4,\\
& \partial_t F-\Delta F+{\rm div}u=\mathcal{N}_3+\mathcal{N}_4.
\end{aligned}\right.
\end{equation}

\begin{lem}
For the strong solutions to the problem \eqref{I-2}--\eqref{I--2}, it holds that
\begin{align}\label{G3.55}
&\frac{\mathrm{d}}{\mathrm{d}t}\Big(\|\rho(t)\|_{L^{2}}^{2}+\|{u}(t)\|_{L^{2}}^{2}+\|F(t)\|_{L^2}+\frac{1}{2}\|G(t)\|_{L^2}^2\Big)\nonumber\\
&+2\|\nabla u(t)\|_{L^2}^2+2\|{\rm div}u(t)\|_{L^2}^2+\|G(t)\|_{L^2}^2
+\|\nabla G(t)\|_{L^2}^2+2\|\nabla F(t)\|_{L^2}^2\nonumber\\
&\,\quad \leq C\delta\Big(\|\nabla^2(u,\theta)(t)\|_{L^2}^2+\|\nabla(\rho,\eta)(t)\|_{L^2}^2+\|\nabla^3\theta(t)\|_{L^2}^2        \Big).     
\end{align}
\end{lem}

\begin{proof}
It follows from \eqref{G3.54} that  
\begin{align}\label{G3.56}
&\frac{1}{2}\frac{\mathrm{d}}{\mathrm{d}t}\Big(\|\rho\|_{L^{2}}^{2}+\|{u}\|_{L^{2}}^{2}+\|F\|_{L^2}+\frac{1}{2}\|G\|_{L^2}^2\Big)\nonumber\\
&+\|\nabla u\|_{L^2}^2+2\|{\rm div}u\|_{L^2}^2+\frac{5}{2}\|G\|_{L^2}^2+\frac{1}{2}\|\nabla G\|_{L^2}^2+\|\nabla F\|_{L^2}^2\nonumber\\
&\,\quad =-2\langle{\rm div}u,G\rangle+\langle\rho,\mathcal{N}_1\rangle+\langle u,\mathcal{N}_2\rangle+\langle F,\mathcal{N}_3+\mathcal{N}_4\rangle+\frac{1}{2}\langle G,4\mathcal{N}_3-\mathcal{N}_4\rangle.
\end{align}
Using Lemma \ref{L2.1}, integration by parts, H\"{o}lder’s
and Young's inequalities, the first three terms on the right hand-side of \eqref{G3.56}
can be estimated as follows:
\begin{align}\nonumber
-2\langle {\rm div}u,G\rangle\leq&\, 2\|G\|_{L^2}^2+\frac{1}{2}\|{\rm div}u\|_{L^2}^2,\nonumber\\ 
\langle\rho,\mathcal{N}_1\rangle\leq&\, C\|\rho\|_{L^3}\|{\rm div}u\|_{L^2}\|\rho\|_{L^6}+C\|\rho\|_{L^3}\|u\|_{L^6}\|\nabla\rho\|_{L^2}\nonumber\\
\leq&\,C\delta\|\nabla(\rho,u)\|_{L^2}^2,\nonumber\\
\langle u,\mathcal{N}_2\rangle\leq&\, C|\langle u,u\cdot\nabla u\rangle|+C|\langle u,g(\rho)\nabla \rho\rangle |+C|\langle u,h(\rho)\theta\nabla \rho\rangle |\nonumber\\
&\,+C|\langle u,g(\rho)\Delta u\rangle |+C|\langle u,g(\rho)\nabla{\rm div} u\rangle |+C|\langle u,g(\rho)\nabla \eta\rangle |\nonumber\\
\leq&\, C\|u\|_{L^3}\|u\|_{L^6}\|\nabla u\|_{L^2}+C\|u\|_{L^6}\|g(\rho)\|_{L^3}\|\nabla \rho\|_{L^2}\nonumber\\
&+\,C\|u\|_{L^6}\|h(\rho)\|_{L^{\infty}}\|\theta\|_{L^3}\|\nabla \rho\|_{L^2}+C\|u\|_{L^6}\|g(\rho)\|_{L^3}\|\Delta u\|_{L^2}\nonumber\\
&+\,C\|u\|_{L^6}\|g(\rho)\|_{L^3}\|\nabla{\rm div} u\|_{L^2}+C\|u\|_{L^6}\|g(\rho)\|_{L^3}\|\nabla \eta\|_{L^2}\nonumber\\
\leq&\,C\delta\big(\|\nabla(\rho,u,\eta)\|_{L^2}^2+\|\nabla^2 u\|_{L^2}^2  \big).
\end{align}
Similarly,
\begin{align}\nonumber
\langle F,\mathcal{N}_3+\mathcal{N}_4\rangle\leq&\, C\|F\|_{L^6}\|g(\rho)\|_{L^3}\|\Delta \theta\|_{L^2}+C\|F\|_{L^6}\|\theta\|_{L^3}\|{\rm div}u\|_{L^2}\nonumber\\
&+\,C\|F\|_{L^6}\|u\|_{L^3}\|\nabla\theta\|_{L^2}+C\|F\|_{L^6}\|g(\rho)\|_{L^3}\|G\|_{L^2}\nonumber\\
&+\,C\|h(\rho)_{L^{\infty}}\|F\|_{L^6}\|\Big(\|{\rm div}u\|_{L^2}\|{\rm div}u\|_{L^3}+\|\nabla u\|_{L^2}\|\nabla u\|_{L^3}\Big)\nonumber\\
&+\,C\|\rho\|_{L^6}\|\theta\|_{L^6}\|F\|_{L^6}\|\theta\|_{L^2}\Big(1+\|\theta\|_{L^{\infty}}+\|\theta\|_{L^{\infty}}^2     \Big)\nonumber\\
\leq&\,C\delta\big(\|\nabla(\rho,u)\|_{L^2}^2+\|G\|_{L^2}^2+\|\nabla^2\theta\|_{L^2}^2                     \big),\nonumber\\
\langle G,4\mathcal{N}_3-\mathcal{N}_4\rangle\leq&\, C\|G\|_{L^2}\|g(\rho)\|_{L^3}\|\Delta \theta\|_{L^6}+C\|G\|_{L^2}\|\theta\|_{L^3}\|{\rm div}u\|_{L^6}\nonumber\\
&+\,C\|G\|_{L^2}\|u\|_{L^3}\|\nabla\theta\|_{L^6}+C\|G\|_{L^2}\|\eta\|_{L^3}\|G\|_{L^6}\nonumber\\
&+\,C\|h(\rho)_{L^{\infty}}\|G\|_{L^2}\|\Big(\|{\rm div}u\|_{L^6}\|{\rm div}u\|_{L^3}+\|\nabla u\|_{L^6}\|\nabla u\|_{L^3}\Big)\nonumber\\
&+\,C\big(\|h(\rho)\|_{L^{\infty}}+1\big)\|\theta\|_{L^6}\|\theta\|_{L^3}\|G\|_{L^2}\Big(1+\|\theta\|_{L^{\infty}}+\|\theta\|_{L^{\infty}}^2     \Big)\nonumber\\
\leq&\,C\delta\big(\|\nabla G\|_{L^2}^2+\|G\|_{L^2}^2+\|\nabla^2(u,\theta)\|_{L^2}^2+\|\nabla^3\theta\|_{L^2}^2                     \big),
\end{align}
substituting all the above estimates 
into \eqref{G3.56}, we obtain the
desired inequality \eqref{G3.55}.
\end{proof}

Then, we further give the estimate of $\nabla \rho$.
\begin{lem}
For the strong solutions to the problem \eqref{I-2}--\eqref{I--2}, it holds that
\begin{align}\label{G3.59}
&\frac{\mathrm{d}}{\mathrm{d}t}\langle u(t),\nabla\rho^l(t)\rangle+\frac{1}{2}\|\nabla\rho(t)\|_{L^2}^2\nonumber\\
&\,\quad \leq 2\|\nabla\rho^H(t)\|_{L^2}^2+\frac{1}{2}\|\nabla u(t)\|_{L^2}^2+4\|{\rm div}u(t)\|_{L^2}^2+2\|\nabla F(t)\|_{L^2}^2\nonumber\\
&\,\quad \ \ \ +C\delta\big(\|\nabla^2(\rho,u)(t)\|_{L^2}^2+\|\nabla u(t)\|_{L^2}^2    \big).
\end{align}
\end{lem}
\begin{proof}
Multiplying \eqref{G3.54}$_2$ by $\nabla \rho^l$ and then integrating the result  over $\mathbb{R}^3$, we have  
\begin{align}\label{G3.60}
&\frac{\mathrm{d}}{\mathrm{d}t}\langle {u},\nabla\rho^{l}\rangle+\langle\nabla\rho^{l},\nabla\rho\rangle\nonumber\\
&\,\quad =\langle {u},\nabla\partial_{t}\rho^{l}\rangle+\langle\nabla\rho^{l},\Delta {u}\rangle+2\langle\nabla\rho^{l},\nabla\mathrm{divu}\rangle-\langle\nabla F,\nabla\rho^l\rangle+\langle\nabla\rho^{l},\mathcal{N}_{2}\rangle.    
\end{align}
Thanks to the definition of \eqref{G2.4} and Young's inequality, we obtain
\begin{align}\label{G3.61}
\langle\nabla\rho^{l},\nabla\rho\rangle=&\|\nabla\rho\|_{L^{2}}^{2}-\langle\nabla\rho^{H},\nabla\rho\rangle 
\geq \frac{7}{8}\|\nabla\rho\|_{L^{2}}^{2}-2\|\nabla\rho^H\|_{L^2}^2.   
\end{align}
Applying the operator $\phi_0(D_x)$ to the \eqref{G3.54}$_1$ yields
\begin{align}\label{G3.62}
\partial_t\rho^l=-{\rm div}{u}^l+\mathcal{N}_1^l.   
\end{align}
By utilizing \eqref{G3.62}, Lemma \ref{L2.5}, integration by parts and Young's inequality, we derive that 
\begin{align}\label{G3.63}
\langle{u},\nabla\partial_{t}\rho^{l}\rangle=&-\langle{\rm div}u,\partial_{t}\rho^{l}\rangle\nonumber\\
=&\,\langle{\rm div}u,\mathrm{divu}^{l}\rangle-\langle\mathrm{divu},\mathcal{N}_{1}^{l}\rangle\nonumber\\
\leq&\,\frac{1}{2}\|{\rm div}u^l\|_{L^{2}}^{2}+\|{\rm div}u\|_{L^{2}}^{2}+\frac{1}{2}\|\mathcal{N}_1^l\|_{L^2}^2\nonumber\\  
\leq&\,\frac{3}{2}\|{\rm div}u\|_{L^{2}}^{2}+C\delta\big(\|\nabla^2(\rho,u)\|_{L^2}^2+ \|\nabla(\rho,u)\|_{L^2}^2   \big).
\end{align}
Similarly, we have
\begin{align}\label{G3.64}
&\langle\Delta {u},\nabla\rho^{l}\rangle+2\langle\nabla{\rm div}u,\nabla\rho^{l}\rangle\nonumber\\
\leq&\,|\langle\nabla{u},\nabla^2\rho^{l}\rangle|+2|\langle\mathrm{divu},\Delta\rho^{l}\rangle|\nonumber\\
\leq&\,\frac{1}{2}\|\nabla u\|_{L^2}^2+2r_0^2\|\nabla\rho\|_{L^2}^2+2\|{\rm div}u\|_{L^2}^2\nonumber\\
\leq&\,\frac{1}{2}\|\nabla u\|_{L^2}^2+2\|{\rm div}u\|_{L^2}^2+\frac{1}{8}\|\nabla\rho\|_{L^2}^2,
\end{align}
where we have used the fact $r_0<\frac{1}{4}$,
It follows from Young's inequality that
\begin{align}
-\langle\nabla F,\nabla\rho^l\rangle\leq &\,\frac{1}{8}\|\nabla\rho^l\|_{L^2}^2+2\|\nabla F\|_{L^2}^2,\label{G3.65}\\
\langle\nabla\rho^{l},\mathcal{N}_{2}\rangle 
\leq&\, \frac{1}{8}\|\nabla\rho^l\|_{L^2}^2+2\|\mathcal{N}_2\|_{L^2}^2\nonumber\\
\leq&\,\frac{1}{8}\|\nabla\rho^l\|_{L^2}^2+C\delta\Big(\|\nabla(\rho,u)\|_{L^2}^2+\|\nabla^2 u\|_{L^2}^2         \Big).\label{G3.66}
\end{align}
Putting \eqref{G3.61}--\eqref{G3.66} into \eqref{G3.60}, we consequently
get the inequality \eqref{G3.59}.
\end{proof}

Then, adding  \eqref{G3.55} and $\frac{1}{4}\times$\eqref{G3.59} up provides
\begin{align}\label{G3.67}
&\frac{\mathrm{d}}{\mathrm{d}t}\Big(\|\rho(t)\|_{L^{2}}^{2}+\|{u}(t)\|_{L^{2}}^{2}+\|F(t)\|_{L^2}+\frac{1}{2}\|G(t)\|_{L^2}^2+\frac{1}{4}\langle u(t),\nabla\rho^l(t)\rangle\Big)\nonumber\\
&+\frac{1}{8}\|\nabla\rho(t)\|_{L^2}^2+\|\nabla u(t)\|_{L^2}^2+\|{\rm div}u(t)\|_{L^2}^2+\|G(t)\|_{L^2}^2
+\|\nabla G(t)\|_{L^2}^2+\|\nabla F(t)\|_{L^2}^2\nonumber\\
&\,\quad \leq C\delta\Big(\|\nabla^2(\rho,u,\theta)(t)\|_{L^2}^2+\|\nabla(\rho,\eta)(t)\|_{L^2}^2+\|\nabla^3\theta(t)\|_{L^2}^2        \Big)+\frac{1}{2}\|\nabla\rho^H(t)\|_{L^2}^2.     
\end{align}
Now, we continue to define the temporal energy functional $\mathcal{L}(t)$ as
\begin{align}\label{G3.68}
\mathcal{L}(t):=\|(\rho,u,F)(t)\|_{L^{2}}^{2}+\frac{1}{2}\|G(t)\|_{L^{2}}^{2}+\frac{1}{4}\langle u(t),\nabla\rho^l(t)\rangle\rangle.  
\end{align}
According to Young's inequality and Lemma \ref{L2.5}, we infer that
\begin{align*}
\frac{1}{4}\langle u(t),\nabla\rho^l(t)\rangle\rangle
\leq&\, \frac{1}{8}\| u(t)\|_{L^2}^2+ \frac{1}{8}\|\nabla\rho^l(t)\|_{L^2}^2\nonumber\\
\leq&\,\frac{1}{8}\| u(t)\|_{L^2}^2+\frac{r_0^2}{8}\| \rho(t)\|_{L^2}^2\nonumber\\
\leq&\,\frac{1}{8}\| u(t)\|_{L^2}^2+\frac{1}{8}\|\rho(t)\|_{L^2}^2,
\end{align*}
where we have used the fact $r_0<\frac{1}{4}$.
Thus,
there exists a positive constant $C_{3}$,  
independent of $\delta$, such that
\begin{align}\label{G3.70}
\frac{1}{C_{3}}\|({\rho},{u},G,F)(t)\|_{L^{2}}^{2}\leq \mathcal{H}(t)\leq C_{3}\|({\rho},{u},G,F)(t)\|_{L^{2}}^{2}.
\end{align}
It is clear that
\begin{align*}
\theta=\frac{1}{5}(G+F),\quad\eta=\frac{1}{5}(4F-G),    
\end{align*}
which means there exists a positive constant $C_4$ such that
\begin{align}\label{G3.72}
\frac{1}{C_4}\Big(\|\theta\|_{L^2}^2+\|\eta\|_{L^2}^2  \Big) \leq \|G\|_{L^2}^2+\|F\|_{L^2}^2\leq {C_4}\Big(\|\theta\|_{L^2}^2+\|\eta\|_{L^2}^2  \Big).   
\end{align}
With the aid of \eqref{G3.70} and \eqref{G3.72}, we conclude that
\begin{align*}
\mathcal{L}(t)\backsim   \|({\rho},{u},\theta,\eta)(t)\|_{L^{2}}^{2} . 
\end{align*}
This together with \eqref{G3.67} yields
the following proposition:
\begin{prop}\label{P3.7}
For the strong solutions to the problem \eqref{I-2}--\eqref{I--2}, it holds that
\begin{align}\label{G3.74}
&\frac{\mathrm{d}}{\mathrm{d}t}\mathcal{L}(t)+\lambda_4\Big(\|\nabla(\rho,u,\theta,\eta)(t)\|_{L^2}^2+\|{\rm div}u(t)\|_{L^2}^2+\|(4\theta-\eta)(t)\|_{L^2}^2
\Big)\nonumber\\
&\,\quad \leq C\delta\Big(\|\nabla^2(\rho,u,\theta)(t)\|_{L^2}^2+\|\nabla(\rho,\eta)(t)\|_{L^2}^2+\|\nabla^3\theta(t)\|_{L^2}^2        \Big)+\frac{1}{2}\|\nabla\rho^H(t)\|_{L^2}^2,     
\end{align}
for some $\lambda_4>0$.
\end{prop}
With Proposition \ref{P3.4} and Proposition \ref{P3.7} in hand,
we obtain the following lemma.
\begin{lem}\label{L3.8}
For the strong solutions to the problem \eqref{I-2}--\eqref{I--2}, it holds that
\begin{align}\label{G3.75}
&\|(\rho,{u},\theta,\eta)(t)\|_{H^{2}}^{2}+\int_{0}^{t}\Big(\|\nabla(\rho,u,\theta)(\tau)\|_{H^{2}}^{2}+\|\nabla \rho(\tau)\|_{H^{1}}^{2}+\|(4\theta-\eta)(\tau)\|_{H^2}^2\Big)\mathrm{d}\tau\nonumber\\
&\,\quad \leq C\|(\rho,{u},\theta,\eta)(0)\|_{H^{2}}^{2}+C_{5}R_0^4\int_{0}^{t}\|(\rho^m,u^m,\theta^m,\eta^m)(\tau)\|_{L^2}^2\mathrm{d}\tau, 
\end{align}
where $C_5$ is independent of
$r_0$, $R_0$ and $T$.
\end{lem}

\begin{proof}
Combining \eqref{G3.49} with \eqref{G3.74}, we get
\begin{align}\label{G3.76}
&\frac{\mathrm{d}}{\mathrm{d}t}\|(\rho,u,\theta,\eta)\|_{H^2}^2+\lambda_5\Big(\|\nabla(u,\theta,\eta)(t)\|_{H^2}^2+\|\nabla \rho(t)\|_{H^1}^2+\|(4\theta-\eta)(t)\|_{H^2}^2
\Big)\nonumber\\
&\,\quad \leq C_6\Big(\|\nabla\rho^H(t)\|_{L^2}^2+\|\nabla^2\rho^L(t)\|_{L^2}^2     +\|\nabla^2(u,\theta,\eta)(t)\|_{L^2}^2\Big),      
\end{align}
where $\lambda_5=\min\{\lambda_3,\lambda_4 \}$ and $C_6$ is independent of
$r_0$, $R_0$ and $T$.
Here, we have used the fact that
\begin{align*}
\|({\rho},{u},\theta,\eta)\|_{H^2}\backsim\|({\rho},{u},\theta,\eta)\|_{L^2}+\|\nabla^2({\rho},{u},\theta,\eta)\|_{L^2}.  
\end{align*}
By using the definition of $r_0$ and $R_0$ in \eqref{G2.2}--\eqref{G2.3}, Lemma \ref{L2.1} and Lemma \ref{L2.5}, we easily obtain
\begin{align}
\|\nabla^{2}{\rho}^{L}\|_{L^{2}}\leq&\,\|\nabla^{2}{\rho}^{l}\|_{L^{2}}+\|\nabla^{2}{\rho}^{m}\|_{L^{2}}\nonumber\\
\leq&\, r_{0}\|\nabla{\rho}^{l}\|_{L^{2}}+R_{0}^{2}\|{\rho}^{m}\|_{L^{2}}\nonumber\\
\leq&\, r_0\|\nabla{\rho}\|_{L^{2}}+R_{0}^{2}\|{\rho}^{m}\|_{L^{2}}\nonumber\\
\leq&\, \sqrt{\frac{\lambda_5}{2 C_6}}\|\nabla{\rho}\|_{L^{2}}+R_{0}^{2}\|{\rho}^{m}\|_{L^{2}},   \label{G3.78}\\
\|\nabla{\rho}^{H}\|_{L^{2}}\leq&\,\|\nabla{\rho}^{m}\|_{L^{2}}+\|\nabla{\rho}^{h}\|_{L^{2}}\nonumber\\
\leq&\, R_{0}\|{\rho}^{m}\|_{L^{2}}+\frac{1}{R_{0}}\|\nabla^{2}{\rho}^{h}\|_{L^{2}}\nonumber\\
\leq&\, R_{0}\|{\rho}^{m}\|_{L^{2}}+\frac{1}{R_{0}}\|\nabla^{2}{\rho}\|_{L^{2}}\nonumber\\
\leq&\, R_{0}^2\|{\rho}^{m}\|_{L^{2}}+\sqrt{\frac{\lambda_5}{2 C_6}}\|\nabla^{2}{\rho}\|_{L^{2}}. \label{G3.79}
\end{align}
Similarly,
\begin{align}\label{G3.80}
\|\nabla^{2}(u,\theta,\eta)\|_{L^{2}}\leq&\,\|\nabla^{2}(u^{l},\theta^{l},\eta^{l})\|_{L^{2}}+\|\nabla^{2}(u^{m},\theta^{m},\eta^{m})\|_{L^{2}}+\|\nabla^{2}(u^{h},\theta^{h},\eta^{h})\|_{L^{2}}\nonumber\\
\leq&\,r_{0}\|\nabla(u,\theta,\eta)\|_{L^{2}}+R_{0}^{2}\|(u^{m},\theta^{m},\eta^{m})\|_{L^{2}}+\frac{1}{R_{0}}\|\nabla^{3}(u,\theta,\eta)\|_{L^{2}}\nonumber\\
\leq&\,\sqrt{\frac{\lambda_5}{2 C_6}}\Big(\|\nabla(u,\theta,\eta)\|_{L^{2}}+\|\nabla^{3}(u,\theta,\eta)\|_{L^{2}}\Big)+R_{0}^{2}\|(u^{m},\theta^{m},\eta^{m})\|_{L^{2}}.   
\end{align}
Here we have chosen  $r_0$ and $ R_0$   satisfying
\begin{gather}
 0<r_0\leq \min\Big\{\frac{1}{4},\sqrt{\frac{\lambda_5}{2C_6}}        \Big\},   \label{G2.2}\\
 R_0> \max\Big\{1,\sqrt{\frac{\lambda_5}{2C_6}},2\sqrt{\frac{C_1}{\lambda_3}} \Big\}. \label{G2.3}
\end{gather}

Plugging the estimates \eqref{G3.78}--\eqref{G3.80} into \eqref{G3.76} and
integrating the resulting inequality with $t$ over $[0,t]$,
we eventually get the desired inequality \eqref{G3.75}. 
\end{proof}

\subsection{Global global existence of strong solutions}

In this section, in order to establish the global existence of strong solutions
to the problem \eqref{I-2}--\eqref{I--2}, we shall study the $L^2$-norm estimates
of medium-frequency part of $(\rho,u,\theta,\eta)$.
Let $\mathbb{A}$ be a matrix of differential operators, which has the following form:
\begin{align*}
\mathbb{A}=\left(
\begin{matrix}
0 & {\rm div} & 0 & 0   \\
\nabla & -\Delta-2\nabla{\rm div} & \nabla & \nabla    \\
0 & {\rm div} & -\Delta+4 & -1   \\
0 & 0 & -4 & 1-\Delta  \\
\end{matrix}\right).
\end{align*}
Then, the corresponding linearized problem to \eqref{I-2}--\eqref{I--2}
reads as  
\begin{equation}\label{G3.82}
\left\{
\begin{aligned}
&\,\partial_t\overline{\mathbb{U}}+\mathbb{A}\overline{\mathbb{U}}=0, \\
&\,\overline{\mathbb{U}}|_{t=0}=\overline{\mathbb{U}}(0),
\end{aligned}\right.
\end{equation}
for $t>0$.
Here, 
\begin{align*}
\overline{\mathbb{U}}(t):=(\overline{\rho}(t),\overline{u}(t),\overline{\theta}(t),\overline{\eta}(t))^{\top},\quad \mathbb{U}(0):=({\rho}_{0},{u}_{0},{\theta}_{0},{\eta}_{0})^{\top}.    
\end{align*}
By performing Fourier transform on \eqref{G3.82} on $x$, we deduce that
\begin{align*}
\overline{\mathbb{U}}(t)=\mathrm{A}(t)\mathbb{U}(0),    
\end{align*}
with $\mathrm{A}(t)=e^{-t\mathbb{A}}$($t\geq 0$) is a semigroup
determined by $\mathbb{A}$ and $\mathcal{F}(\mathrm{A}(t)g):=e^{-t{\mathbb{A}}_{\xi}}\widehat{g}(\xi)$,
where
\begin{align*}
\mathbb{A}_{\xi}=\left(
\begin{matrix}
0 & i\xi^T & 0 & 0   \\
i\xi & |\xi|^2\delta_{ij}+2\xi_i\xi_j & i\xi & i\xi   \\
0 & i\xi^T & |\xi|^2+4 & -1   \\
0 & 0 & -4 & |\xi|^2+1  \\
\end{matrix}\right).   
\end{align*}
In what follows, we give the estimate of strong solutions for the medium-frequency part.
\begin{lem}\label{L3.9}
For any integer $k\geq 0$, it holds that
\begin{align}\label{G3.86}
\|\nabla^k\big(\mathrm{A}(t)\mathbb{U}^m(0)\big)\|_{L^2}\leq C{e}^{-Ct}\|\mathbb{U}(0)\|_{L^2}.   
\end{align}
\end{lem}
\begin{proof}
Taking $r=\frac{r_0}{2}$ and $R=R_0+1$ in Proposition \ref{P5.1}, then
utilizing   \eqref{G5.22} and Plancherel theorem, one has
\begin{align}\label{G3.87}
\big\|\partial_{x}^{\alpha}(\overline{p}^{m},\overline{d}^{m},
\overline{\theta}^{m},\overline{\eta}^{m})(t)\big\|_{L^{2}}
=&\,\big\|(i\xi)^{\alpha}(\widehat{\overline{\rho}^{m}},\widehat{\overline{d}^{m}},
\widehat{\overline{\theta}^{m}},\widehat{\overline{\eta}^{m}})(t)\big\|_{L_{\xi}^{2}}\nonumber\\
\leq &\,C\Big(\int_{\frac{r_{0}}{2}\leq|\xi|\leq R_{0}+1}|\xi|^{2|\alpha|}\big|(\widehat{\overline{\rho}},\widehat{\overline{d}},\widehat{\overline{\theta}},\widehat{\overline{\eta}})(\xi,t)\big|^{2}\mathrm{d}\xi\Big)^{\frac{1}{2}}\nonumber\\
\leq&\, C\mathrm{e}^{-\kappa t}\Big(\int_{\mathbb{R}^{3}}\big|(\widehat{{\rho}},\widehat{{d}},\widehat{{\theta}},\widehat{{\eta}})(\xi,0)\big|^{2}\mathrm{d}\xi\Big)^{\frac{1}{2}}\nonumber\\
\leq&\, C{e}^{-\kappa t}\|({\rho}_{0},{u}_{0},{\theta}_{0},{\eta}_{0})\|_{L^{2}}.   
\end{align}
Similar to \eqref{G3.87}, using \eqref{G5.22}, we have
\begin{align}\label{G3.88}
\|\partial_x^\alpha(\mathcal{P}\overline{u})^m)(t)\|_{L^2}\leq C{e}^{-\kappa_1t}\|{u}(0)\|_{L^2},
\end{align}
for some constant $\kappa_1>0$.
Combining \eqref{G3.87}--\eqref{G3.88} and the identity \eqref{G5.3},
we complete the proof of \eqref{G3.86}.
\end{proof}
Additionally, based on   Lemma \ref{L3.9}, we establish the   strong
solutions to the nonlinear problem \eqref{I-2}--\eqref{I--2}:
\begin{equation}\label{G3.89}
\left\{
\begin{aligned}
&\,\partial_t\overline{\mathbb{U}}+\mathbb{A}\overline{\mathbb{U}}=\mathcal{N}(\mathbb{U}), \\
&\,\overline{\mathbb{U}}|_{t=0}=\overline{\mathbb{U}}(0),
\end{aligned}\right.
\end{equation}
for $t>0$, where
\begin{align*}
\mathbb{U}(t):=(\overline{\rho}(t),\overline{u}(t),\overline{\theta}(t),\overline{\eta}(t))^\top,\quad\text{and}\quad \mathcal{N}(\mathbb{U}):=(\mathcal{N}_{1},\mathcal{N}_{2},\mathcal{N}_{3},\mathcal{N}_4)^{\top}.  
\end{align*}
Through Duhamel's principle, \eqref{G3.89} can
be rewritten as  
\begin{align}\label{G3.91}
\mathbb{U}(t)=\mathrm{A}(t)\mathbb{U}(0)+\int_0^t\mathrm{A}(t-\tau)\mathcal{N}(\mathbb{U})(\tau)\mathrm{d}\tau.    
\end{align}
\begin{lem}\label{L3.10}
For any integer $k\geq 0$, it holds that
\begin{align}\label{G3.92}
\int_{0}^{t}\|\nabla^{k}\mathbb{U}^{m}(\tau)\|_{L^{2}}^{2}\mathrm{d}\tau\leq C\|\mathbb{U}(0)\|_{L^{2}}^{2}+C\delta\int_{0}^{t}\Big(\|\nabla(\rho,\eta)(\tau)\|_{L^{2}}^{2}+\|\nabla(u,\theta)(\tau)\|_{H^{1}}^{2}\Big)\mathrm{d}\tau.   
\end{align}
\end{lem}
\begin{proof}
Applying the operator $(1-\phi_0(D_x)-\phi_1(D_x))\mathbb{I}$ to \eqref{G3.91},
we obtain
\begin{align}\label{G3.93}
\mathbb{U}^m(t)=\mathrm{A}(t)\mathbb{U}^m(0)+\int_0^t\mathrm{A}(t-\tau)\mathcal{N}^m(\mathbb{U})(\tau)\mathrm{d}\tau.  
\end{align}
Then, it follows from Lemma \ref{L3.9} and \eqref{G3.93} that
\begin{align*}
\|\nabla^k\mathbb{U}^m(t)\|_{L^2}\leq C e^{-Ct}\|\mathbb{U}(0)\|_{L^2}+C\int_0^t e^{-C(t-\tau)}\|\mathcal{N}(\mathbb{U})(\tau)\|_{L^2}\mathrm d\tau,    
\end{align*}
which implies that 
\begin{align}\label{G3.95}
\int_{0}^{t}\|\nabla^{k}\mathbb{U}^{m}(\tau)\|_{L^{2}}^{2}\mathrm{d}\tau\leq C\|\mathbb{U}(0)\|_{L^{2}}^{2}+C\int_{0}^{t}\mathrm{d}s\Big(\int_{0}^{s}e^{-C(s-\tau)}\|\mathcal{N}(\mathbb{U})(\tau)\|_{L^{2}}\mathrm{d}\tau\Big)^{2}.   
\end{align}
According to H\"{o}lder’s inequality and exchanging the order of integration,
we derive that
\begin{align}\label{G3.96}
&\int_{0}^{t}\mathrm{d}s\Big(\int_{0}^{s}\mathrm{e}^{-C(s-\tau)}\|\mathcal{N}(\mathrm{U})(\tau)\|_{L^{2}}\mathrm{d}\tau\Big)^{2}\nonumber\\
\leq&\, C\int_{0}^{t}\mathrm{d}s\Big(\int_{0}^{s}e^{-C(s-\tau)}\mathrm{d}\tau\Big)\Big(\int_{0}^{s}{e}^{-C(s-\tau)}\|\mathcal{N}(\mathrm{U})(\tau)\|_{L^{2}}^{2}\mathrm{d}\tau\Big)\nonumber\\
\leq&\, C\int_{0}^{t}\mathrm{d}s\int_{0}^{s}{e}^{-C(s-\tau)}\|\mathcal{N}(\mathrm{U})(\tau)\|_{L^{2}}^{2}\mathrm{d}\tau\nonumber\\
\leq &\,C\int_{0}^{t}\|\mathcal{N}(\mathrm{U})(\tau)\|_{L^{2}}^{2}\mathrm{d}\tau\int_{\tau}^{t}e^{-C(s-\tau)}\mathrm{d}s\nonumber\\
\leq&\, C\int_{0}^{t}\|\mathcal{N}(\mathrm{U})(\tau)\|_{L^{2}}^{2}\mathrm{d}\tau.  
\end{align}
Substituting \eqref{G3.96} into \eqref{G3.95} yields
\begin{align}\label{G3.97}
\int_{0}^{t}\|\nabla^{k}\mathbb{U}^{m}(\tau)\|_{L^{2}}^{2}\mathrm{d}\tau\leq C\|\mathbb{U}(0)\|_{L^{2}}^{2}+C\int_{0}^{t}\|\mathcal{N}(\mathrm{U})(\tau)\|_{L^{2}}^{2}\mathrm{d}\tau.   
\end{align}
By applying Lemmas \ref{L2.1}--\ref{L2.2}, it is obvious that
\begin{align}\label{G3.98}
\|\mathcal{N}(\mathbb{U})(\tau)\|_{L^2}\leq&\, C\|({\rho},{u},\theta)\|_{L^\infty}\|\nabla({\rho},{u})\|_{L^2}+C\|{\rho}\|_{L^\infty}\|\nabla^2({u},\theta)\|_{L^2} \nonumber\\
&\,+C\|g(\rho)\|_{L^{\infty}}\|\nabla\eta\|_{L^2}+C\|{\rm div}u\|_{L^3}\|{\rm div}u\|_{L^6}\nonumber\\
&\,+C\|\rho\|_{L^3}\|\eta\|_{L^6}+C\|\rho\|_{L^3}\|\theta\|_{L^6}+C\|\nabla u\|_{L^3}\|\nabla u\|_{L^6}\nonumber\\
&\,+C\|\theta\|_{L^3}\|\theta\|_{L^6}\Big(1+\|\theta\|_{L^{\infty}}+\|\theta\|_{L^{\infty}}^2      \Big)\nonumber\\
\leq&\, C\delta\Big(\|\nabla(\rho,u,\theta,\eta)\|_{L^2}+\|\nabla^2(u,\theta)\|_{L^2}        \Big).
\end{align}
Putting \eqref{G3.98} into \eqref{G3.97}, we eventually get \eqref{G3.92}.
\end{proof}
Finally, we give the uniform of a priori estimates of strong solutions.
For any $t\in [0,T]$, we define
\begin{align*}
\mathbf{N}(t):=&\,\sup_{0\leq\tau\leq t}\|(\rho,u,\theta,\eta)(\tau)\|_{H^{2}}^{2}\nonumber\\
&\,+\int_{0}^{t}\Big(\|\nabla\rho(\tau)\|_{H^{1}}^{2}+\|\nabla(u,\theta,\eta)(\tau)\|_{H^{2}}^{2}+\|(4\theta-\eta)(\tau)\|_{H^2}^2\Big)\mathrm{d}\tau.   
\end{align*}
This together with the smallness of $\delta$,
Lemma \ref{L3.8} and Lemma \ref{L3.10} yields
\begin{align}\label{nt}
\mathbf{N}(t)\leq C \mathbf{N}(0).    
\end{align}

\begin{proof}[Proof of Theorem \ref{T1.1}]
By the fixed point theorem of
the standard iteration arguments, we directly get the local existence
and uniqueness of strong solutions to system \eqref{I-2}--\eqref{I--2}
(see Section 2.1 in \cite{DD-AMPA-2017}).
For simplicity, we omit some details here.   Combining the above a priori estimates \eqref{nt} and 
the standard 
continuity principle, we can easily obtain the global existence and uniqueness of 
strong solutions.
Hence,   the proof of Theorem \ref{T1.1} is completed.
\end{proof}

\section{Optimal time-decay rates of strong solutions}
In this
section, we shall establish the optimal
time-decay rates of strong solutions to our
system under the assumption of Theorem \ref{T1.2}.
The estimates of low-medium frequency part of $(\rho,u,\theta,\eta)$ 
can be obtained through the spectral analysis on the linearized version of
\eqref{I-2}--\eqref{I--2}. Based on
these estimates, we consequently get the time-decay rates of strong solutions 
to the nonlinear problem \eqref{I-2}--\eqref{I--2}.

\subsection{Estimates on the low-medium frequency part}
\begin{lem}\label{L4.1}
For any integer $k\geq 0$,
it holds that 
\begin{align}\label{G4.1}
\|\nabla^k\big(\mathrm{A}(t)\mathbb{U}^l(0)\big)\|_{L^2}\leq C(1+t)^{-\frac{3}{2}(\frac{1}{p}-\frac{1}{2})-\frac{k}{2}}\|\mathbb{U}(0)\|_{L^p}, \quad   1\leq p\leq 2  .
\end{align}
\end{lem}
\begin{proof}
It follows from \eqref{G2.2}, \eqref{G5.20} and Plancherel theorem, we obtain
that
\begin{align*}
\|\partial_{x}^{\alpha}(\overline{\rho}^l,\overline{d}^l,\overline{\theta}^l,\overline{\eta}^l)(t)\|_{L^{2}}
=&\,\big\|(i\xi)^{\alpha}(\widehat{\overline{\rho}^l},\widehat{\overline{d}^l},
\widehat{\overline{\theta}^l},\widehat{\overline{\eta}^l})(t)\big\|_{L_{\xi}^{2}}
\nonumber\\
=&\,\Big(\int_{\mathbb{R}^{3}}\big|(i\xi)^{\alpha}(\widehat{\overline{\rho}^l},\widehat{\overline{d}^l},\widehat{\overline{\theta}^l},\widehat{\overline{\eta}^l})(\xi,t)\big|^{2}\mathrm{d}\xi\Big)^{\frac{1}{2}}\nonumber\\
\leq&\, C\Big(\int_{|\xi|\leq r_{0}}|\xi|^{2|\alpha|}\big|(\widehat{\overline{\rho}},\widehat{\overline{d}},\widehat{\overline{\theta}},\widehat{\overline{\eta}})(\xi,t)\big|^{2}\mathrm{d}\xi\Big)^{\frac{1}{2}}\nonumber\\
\leq&\, C\Big(\int_{|\xi|\leq r_{0}}|\xi|^{2|\alpha|}{e}^{-C|\xi|^{2}t}\big|(\widehat{{\rho}},\widehat{{d}},\widehat{{\theta}},\widehat{{\eta}})(\xi,0)\big|^{2}\mathrm{d}\xi\Big)^{\frac{1}{2}}.    
\end{align*}
This  together with  H\"{o}lder’s inequality and the Huasdorff-Young inequality
yields
\begin{align}\label{G4.3}
\|\partial_{x}^{\alpha}(\overline{\rho}^l,\overline{d}^l,\overline{\theta}^l,\overline{\eta}^l)(t)\|_{L^{2}}\leq&\, C\|(\widehat{\overline{\rho}},\widehat{\overline{d}},\widehat{\overline{\theta}},\widehat{\overline{\eta}})(0)\|_{L_{\xi}^{q}}(1+t)^{-\frac{3}{2}(\frac{1}{2}-\frac{1}{q})-\frac{|\alpha|}{2}}\nonumber\\
\leq&\, C\|({\rho},{u},{\theta},{\eta})(0)\|_{L^{p}}(1+t)^{-\frac{3}{2}(\frac{1}{p}-\frac{1}{2})-\frac{|\alpha|}{2}}.    
\end{align}
Here, $1\leq p\leq2$ and $\frac{1}{p}+\frac{1}{q}=1$.
Similar to \eqref{G4.3}, it follows from \eqref{G2.2} and \eqref{G5.22} that
\begin{align}\label{G4.4}
\|\partial_x^\alpha(\mathcal{P}\overline{u})^l(t)\|_{L^2}\leq C\|{u}(0)\|_{L^p}(1+t)^{-\frac{3}{2}(\frac{1}{p}-\frac{1}{2})-\frac{|\alpha|}{2}}.   
\end{align}
Therefore, we directly get \eqref{G4.1} by combining \eqref{G4.3} and \eqref{G4.4}.
\end{proof}

\begin{lem}\label{L4.2}
For any integer $k\geq 0$, it holds that
\begin{align}\label{G4.5}
\|\nabla^{k}\mathbb{U}^{L}(t)\|_{L^{2}}\leq&\, C_{7}(1+t)^{-\frac{3}{4}-\frac{k}{2}}\|\mathbb{U}(0)\|_{L^{1}}+C_{7}e^{-Ct}\|\mathbb{U}(0)\|_{L^{2}}\nonumber\\
&\,+C_{7}\int_{0}^{\frac{t}{2}}(1+t-\tau)^{-\frac{3}{4}-\frac{k}{2}}\|\mathcal{N}(\mathbb{U})(\tau)\|_{L^{1}}\mathrm{d}\tau\nonumber\\
&\,+C_{7}\int_{\frac{t}{2}}^{t}(1+t-\tau)^{-\frac{k}{2}}\|\mathcal{N}(\mathbb{U})(\tau)\|_{L^{2}}\mathrm{d}\tau\nonumber\\
&+C_{7}\int_{0}^{t}e^{-C(t-\tau)}\|\mathcal{N}(\mathbb{U})(\tau)\|_{L^{2}}\mathrm{d}\tau, 
\end{align}
for a positive constant $C_7$.
\end{lem}
\begin{proof}
By utilizing Lemma \ref{L3.9}, Lemma \ref{L4.1} and \eqref{G3.91}, 
\eqref{G4.5} can be easily gotten. For simplicity, we omit the details here.
\end{proof}

\subsection{Time-decay rates of strong solutions} 
In this subsection, thanks
to Lemmas \ref{L4.1}--\ref{L4.2}, we can  get the time-decay rates
of $(\rho, u, \theta, \eta)$.

\begin{lem}\label{L4.3}
For the strong solutions to the problem \eqref{I-2}--\eqref{I--2}, it holds that
\begin{align}\label{G4.6}
\|\nabla^m({\rho},{u},\theta,\eta)(t)\|_{L^2}\leq C(1+t)^{-\frac{3}{4}-\frac{m}{2}},    
\end{align}
for $m=0,1,2$.
\end{lem}
\begin{proof}
Denote
\begin{align}\label{G4.7}
\mathcal{M}(t):=\sup\limits_{0\leq \tau\leq t}\sum\limits_{m=0}^{2}(1+\tau)^{\frac{3}{4}+\frac{m}{2}}\|\nabla^{m}({\rho},u,\theta,\eta)(\tau)\|_{L^{2}}.   
\end{align} 
From the definition of \eqref{G4.7}, it is obvious that
\begin{align}\label{G4.8}
\|\nabla^m({\rho},{u},\theta,\eta)(\tau)\|_{L^2}\leq C_8(1+\tau)^{-\frac{3}{4}-\frac{m}{2}}\mathcal{M}(t),   
\end{align}
for $m=0,1,2$ and $0\leq \tau \leq t$,
where the constant $C_8>0$ is independent of $\delta$. 
By using \eqref{G4.8}, H\"{o}lder’s inequality and Lemmas \ref{L2.1}--\ref{L2.2}, we get
\begin{align}
\|\mathcal{N}(\mathbb{U})(\tau)\|_{L^{1}}\leq&\, C\|({\rho},{u},\theta)\|_{L^{2}}\|\nabla({\rho},{u},\eta)\|_{L^{2}}+C\|{\rho}\|_{L^{2}}\|\nabla^{2}(u,\theta)\|_{L^{2}}\nonumber\\
&\,+C\|\nabla u\|_{L^{2}}^2+C\|{\rho}\|_{L^{2}}\|(\theta,\eta)\|_{L^{2}}+C\|\theta\|_{L^2}^2\Big(1+\|\theta\|_{L^{\infty}}+\|\theta\|_{L^{\infty}}^2\Big)\nonumber\\
\leq&\, C\delta^{\frac{2}{3}-\sigma}\mathcal{M}^{\frac{4}{3}+\sigma}(t)(1+\tau)^{-1-\frac{3}{4}\sigma}, 
 \label{G4.9}\\
 \|\mathcal{N}(\mathbb{U})(\tau)\|_{L^{2}}\leq&\, C\|({\rho},{u},\theta)\|_{L^{\infty}}\|\nabla({\rho},{u},\eta)\|_{L^{2}}+C\|{\rho}\|_{L^{\infty}}\|\nabla^{2}(u,\theta)\|_{L^{2}}+C\|\nabla u\|_{H^{1}}\|\nabla^2 u\|_{L^{2}}\nonumber\\
&\,+C\|{\rho}\|_{L^{\infty}}\|\nabla(u,\theta)\|_{L^{2}}+C\|\theta\|_{L^2}\|\nabla\theta\|_{L^2}\Big(1+\|\theta\|_{L^{\infty}}+\|\theta\|_{L^{\infty}}^2\Big)\nonumber\\
\leq\,& C\delta^{\frac{1}{3}-\sigma}
\mathcal{M}^{\frac{5}{3}+\sigma}(t)(1+\tau)^{-\frac{7}{4}-\frac{3}{4}\sigma},   \label{G4.10}
\end{align}
where $0<\sigma<\frac{1}{3}$ is a fixed constant.
By using \eqref{G4.9}--\eqref{G4.10}, Lemma \ref{L2.3} and Lemma \ref{L4.2}, we have
\begin{align}\label{G4.11}
\|\nabla^{m}\mathbb{U}^{L}(t)\|_{L^{2}}\leq&\, C_{7}(1+t)^{-\frac{3}{4}-\frac{m}{2}}\|\mathbb{U}(0)\|_{L^{1}}+C_{7}{e}^{-Ct}\|\mathbb{U}(0)\|_{L^{2}}\nonumber\\
&\,+C\delta^{\frac{2}{3}-\sigma}\mathcal{M}^{\frac{4}{3}+\sigma}(t)\int_{0}^{\frac{1}{2}}(1+t-\tau)^{-\frac{3}{4}-\frac{m}{2}}(1+\tau)^{-1-\frac{3}{4}\sigma}\mathrm{d}\tau\nonumber\\
&\,+C\delta^{\frac{1}{3}-\sigma}\mathcal{M}^{\frac{5}{3}+\sigma}(t)\int_{\frac{t}{2}}^{t}(1+t-\tau)^{-\frac{m}{2}}(1+\tau)^{-\frac{7}{4}-\frac{3}{4}\sigma}\mathrm{d}\tau\nonumber\\
&\,+C\delta^{\frac{1}{3}-\sigma}\mathcal{M}^{\frac{5}{3}+\sigma}(t)\int_{0}^{t}{e}^{-C(t-\tau)}(1+\tau)^{-\frac{7}{4}-\frac{3}{4}\sigma}\mathrm{d}\tau\nonumber\\
\leq&\, C_{7}(1+t)^{-\frac{3}{4}-\frac{m}{2}}\|\mathbb{U}(0)\|_{L^{1}}+C_{7}{e}^{-Ct}\|\mathbb{U}(0)\|_{L^{2}}\nonumber\\
&\,+C\Big(\delta^{\frac{2}{3}-\sigma}\mathcal{M}^{\frac{4}{3}+\sigma}(t)+C\delta^{\frac{1}{3}-\sigma}M^{\frac{5}{3}+\sigma}(t)\Big)(1+t)^{-\frac{3}{4}-\frac{m}{2}},    
\end{align}
for $m=0,1,2$.
In particular, taking $m=2$ in \eqref{G4.11} gives
\begin{align}\label{G4.12}
\|\nabla^{2}\mathbb{U}^{L}(t)\|_{L^{2}}&\leq C_{8}(1+t)^{-\frac{7}{4}}\|\mathbb{U}(0)\|_{L^{1}\cap L^{2}}\nonumber\\
&+C_{8}\Big(\delta^{\frac{2}{3}-\sigma}\mathcal{M}^{\frac{4}{3}+\sigma}(t)+\delta^{\frac{1}{3}-\sigma}\mathcal{M}(t)^{\frac{5}{3}+\sigma}\Big)(1+t)^{-\frac{7}{4}},   
\end{align}
for some constant $C_8>0$, which is independent of $\delta$.
Then, it follows from Proposition \ref{P3.4} and Lemma \ref{L2.5} that
\begin{align*}
&\frac{\mathrm{d}}{\mathrm{d}t}\mathcal{H}(t)+\lambda_3\Big(\|\nabla^{2}(\theta,\eta,\rho)(t)\|_{L^{2}}^{2}+\|\nabla^{2}\mathrm{div}u(t)\|_{L^{2}}^{2}  \Big)\nonumber\\
&+\lambda_3\Big(\frac{1}{2}\|\nabla^{3}(u,\theta,\eta)(t)\|_{L^{2}}^{2}+\frac{1}{2} R_0^2\|\nabla^{2}(u^h,\theta^h,\eta^h)(t)\|_{L^{2}}^{2} \Big) \nonumber\\
&\,\quad \leq C_1\Big(\|\nabla^{2}(u,\theta)(t)\|_{L^{2}}^{2}+\|\nabla^2\rho^L(t)\|_{L^2}^2+\|\nabla^2\eta^h(t)\|_{L^2}^2\Big),
\end{align*}
this together with \eqref{G2.3} yields
\begin{align*}
&\frac{\mathrm{d}}{\mathrm{d}t}\mathcal{H}(t)+\lambda_3\Big(\|\nabla^{2}(\theta,\eta,\rho)(t)\|_{L^{2}}^{2}+\|\nabla^{2}\mathrm{div}u(t)\|_{L^{2}}^{2}  \Big)\nonumber\\
&+\lambda_3\Big(\frac{1}{2}\|\nabla^{3}(u,\theta,\eta)(t)\|_{L^{2}}^{2}+\frac{1}{4} R_0^2\|\nabla^{2}(u^h,\theta^h,\eta^h)(t)\|_{L^{2}}^{2} \Big) \nonumber\\
&\,\quad \leq C_1\|\nabla^{2}(\rho^L,u^L,\theta^L)(t)\|_{L^{2}}^{2}.
\end{align*}
Adding $\frac{1}{4}R_0^2\lambda_3\|\nabla^{2}(u^L,\theta^L,\eta^L)(t)\|_{L^{2}}^{2}$
on the both sides of \eqref{G3.14}, we have
\begin{align*}
\frac{\mathrm{d}}{\mathrm{d}t}\mathcal{H}(t)+C_{9}\mathcal{H}(t)\leq \Big({C}_{1}+\frac{1}{4}R_0^2\lambda_3\Big)\|\nabla^{2}(\rho^L,u^L,\theta^L)(t)\|_{L^{2}}^{2},    
\end{align*}
for some positive constant $C_9$.
By using Gronwall's inequality, we arrive at
\begin{align}\label{G4.16}
\mathcal{H}(t)\leq\mathrm e^{-C_9t}\mathcal{H}(0)+C\int_0^t e^{-C_9(t-\tau)}\|\nabla^{2}(\rho^L,u^L,\theta^L)(\tau)\|_{L^2}^2\mathrm d\tau.    
\end{align}
Substituting \eqref{G4.12} into \eqref{G4.16}, we infer that
\begin{align}\label{G4.17}
\mathcal{H}(t)\leq&\,{e}^{-C_{9}t}\mathcal{H}(0)+C\|\mathbb{U}(0)\|_{L^{1}\cap L^{2}}^{2}\int_{0}^{t}{e}^{-C_{9}(t-\tau)}(1+\tau)^{-\frac{7}{2}}\mathrm{d}\tau\nonumber\\
&\,+C\Big(\delta^{\frac{4}{3}-2\sigma}\mathcal{M}^{\frac{8}{3}+2\sigma}(t)+\delta^{\frac{2}{3}-2\sigma}\mathcal{M}^{\frac{10}{3}+2\sigma}(t)\Big)\int_{0}^{t}{e}^{-C_{9}(t-\tau)}(1+\tau)^{-\frac{7}{2}}\mathrm{d}\tau\nonumber\\
\leq&\,{e}^{-C_{9}t}\mathcal{H}(0)+C\|\mathbb{U}(0)\|_{L^{1}\cap L^{2}}^{2}(1+t)^{-\frac{7}{2}}\nonumber\\
&\,+C\Big(\delta^{\frac{4}{3}-2\sigma}\mathcal{M}^{\frac{8}{3}+2\sigma}(t)+\delta^{\frac{2}{3}-2\sigma}\mathcal{M}^{\frac{10}{3}+2\sigma}(t)\Big)(1+t)^{-\frac{7}{2}}.   
\end{align}
With the help of Lemma \ref{L2.5} and \eqref{G3.50}, it's clear that
\begin{align}\label{G4.18}
\|\nabla^{m}\mathbb{U}(t)\|_{L^{2}}^{2}
\leq&\, C\|\nabla^{m}\mathbb{U}^{L}(t)\|_{L^{2}}^{2}+C\|\nabla^{m}\mathbb{U}^{h}(t)\|_{L^{2}}^{2}\nonumber\\
\leq&\, C\|\nabla^{m}\mathbb{U}^{L}(t)\|_{L^{2}}^{2}+C\|\nabla^{2}\mathbb{U}(t)\|_{L^{2}}^{2}\nonumber\\
\leq&\, C\|\nabla^{m}\mathbb{U}^{L}(t)\|_{L^{2}}^{2}+C\mathcal{H}(t),    
\end{align}
for $m=0,1,2$.
Combining \eqref{G4.11} with \eqref{G4.17}--\eqref{G4.18}, we end up with
\begin{align*}
\|\nabla^{m}\mathbb{U}(t)\|_{L^{2}}^{2}\leq&\,{e}^{-C_{9}t}\mathcal{H}(0)+C\|\mathbb{U}(0)\|_{L^{1}\cap L^{2}}^{2}(1+t)^{-\frac{3}{2}-m}\nonumber\\
&+C\Big(\delta^{\frac{4}{3}-2\sigma}\mathcal{M}^{\frac{8}{3}+2\sigma}(t)+\delta^{\frac{2}{3}-2\sigma}\mathcal{M}^{\frac{10}{3}+2\sigma}(t)\Big)(1+t)^{-\frac{3}{2}-m}.    
\end{align*}
Then, by using the definition of $\mathcal{M}(t)$ and the smallness of $\delta$,
we have
\begin{align}\label{G4.20}
\mathcal{M}(t)\leq C_{10}\Big(\|({\rho},{u},\theta,\eta)(0)\|_{L^1\cap H^2}+\delta^{\frac{2}{3}-\sigma}\mathcal{M}^{\frac{4}{3}+\sigma}(t)+\delta^{\frac{1}{3}-\sigma}\mathcal{M}^{\frac{5}{3}+\sigma}(t)\Big)    
\end{align}
for some constant $C_{10}>0$, which is independent of $\delta$.
Thanks to Young's inequality, \eqref{G4.20} can be rewritten as
\begin{align}\label{G4.21}
\mathcal{M}(t)\leq C_0+C_\delta \mathcal{M}^2(t),    
\end{align}
where
\begin{align*}
C_0:=C_{10}\|({\rho},{u},\theta,\eta)(0)\|_{L^1\cap H^2}+\frac{2-3\sigma}{6}C_{10}^{\frac{6}{2-3\sigma}}+\frac{1-3\sigma}{6}C_{10}^{\frac{6}{1-3\sigma}},
\end{align*}
and
\begin{align*}
C_\delta:=\frac{4+3\sigma}{6}\delta^{\frac{4-6\sigma}{4+3\sigma}}+\frac{5+3\sigma}{6}\delta^{\frac{2-6\sigma}{5+3\sigma}}.    
\end{align*}

Now, we claim that $\mathcal{M}(t)\leq C$. Assume that there exists a constant $t^*>0$ such that  $\mathcal{M}(t^*)>2C_0$.
Noticing that $\mathcal{M}(0)=\|(\rho,u,\theta,\eta)\|_{H^2}$ is sufficiently small
and $\mathcal{M}(t)\in C^0([0,+\infty))$, thus there exists $t_1\in(0,t^*)$ such
that $\mathcal{M}(t_1)=2C_0$.

According to the \eqref{G4.21}, it yields
\begin{align*}
\mathcal{M}(t_1)\leq C_0+C_\delta \mathcal{M}^2(t_1),  
\end{align*}
which means
\begin{align}\label{G4.25}
\mathcal{M}(t_1)\leq \frac{C_0}{1-C_\delta \mathcal{M}(t_1)}.    
\end{align}
Since $\delta$ is a small positive constant, we can choose it satisfying $C_{\delta}<\frac{1}{4C_{0}}$,
this together with \eqref{G4.25} yields
\begin{align*}
 \mathcal{M}(t_{1})<2C_{0},  
\end{align*}
which leads to a contradiction with $\mathcal{M}(t_1)=2C_0$.
As a result, $\mathcal{M}(t)\leq 2C_0$ for all $t\in[t^*,\infty)$.
Thanks to the continuity of $\mathcal{M}(t)$, we get
$\mathcal{M}(t)\leq C$ for all $t\in[0,+\infty)$. Using
the definition of $\mathcal{M}(t)$, we obtain the inequality \eqref{G4.6}.
\end{proof}

\subsection{Proof of Theorem \ref{T1.2}}
\begin{proof}[Proof of Theorem \ref{T1.2}]
With the help of Lemma \ref{L4.3}, we   continue to prove Theorem \ref{T1.2}.
By a directly calculation, we derive that
\begin{align*}
\|\partial_{t}\rho(t)\|_{L^{2}}\leq&\, C\|\mathrm{div}u(t)\|_{L^{2}}+C\|\mathcal{N}_{1}(t)\|_{L^{2}}\nonumber\\
\leq&\,C\|\nabla u(t)\|_{L^{2}}+\|\rho(t)\|_{L^{\infty}}\|\nabla u(t)\|_{L^{2}}+\|u(t)\|_{L^{\infty}}\|\nabla \rho(t)\|_{L^{2}}\nonumber\\
\leq&\,C(1+t)^{-\frac{5}{4}}.    
\end{align*}
Similarly, it holds 
\begin{align*}
\|\partial_{t} u(t)\|_{L^{2}}\leq&\, C\|\nabla \rho(t)\|_{L^{2}}+C\|\nabla (\theta,\eta)(t)\|_{L^{2}}+C\|\Delta u(t)\|_{L^2}+C\|{\rm div} u(t)\|_{L^2}+C\|\mathcal{N}_{2}(t)\|_{L^{2}}\nonumber\\
\leq&\,C(1+t)^{-\frac{5}{4}},    
\end{align*}
and
\begin{align*}
\|\partial_{t}(\theta,\eta)(t)\|_{L^{2}}\leq&\,C\|\Delta(\theta,\eta)(t)\|_{L^{2}}+C\|(\theta,\eta)(t)\|_{L^{2}}+\|{\rm div}u(t)\|_{L^{2}}+C\|(\mathcal{N}_3,\mathcal{N}_4)(t)\|_{L^{2}}\nonumber\\
\leq&\,C(1+t)^{-\frac{3}{4}}.   
\end{align*}
Finally,
   we give the
proofs of \eqref{G1.9}--\eqref{G1.10}. 
From \eqref{G4.6} and Lemma \ref{L2.1}, we have
\begin{align}\label{G4.30}
&\|(\rho,u,\theta,\eta)\|_{L^6}\leq C\|\nabla (\rho,u,\theta,\eta)\|_{L^2}\leq C(1+t)^{-\frac{5}{4}},\\\label{G4.31} 
&\|(\rho,u,\theta,\eta)\|_{L^2}\leq C(1+t)^{-\frac{3}{4}},\\ \label{G4.32}
&\|\nabla(\rho,u,\theta,\eta)\|_{L^6}\leq C\|\Delta (\rho,u,\theta,\eta)\|_{L^2}\leq C(1+t)^{-\frac{7}{4}}.
\end{align}
For $p\in [2,6]$, combining Lemma \ref{L2.4} and \eqref{G4.30}-\eqref{G4.31} implies that
\begin{align*}
&\|(\rho,u,\theta,\eta)\|_{L^p}\leq C\|(\rho,u,\theta,\eta)\|_{L^2}^{\zeta}\|(\rho,u,\theta,\eta)\|_{L^6}^{{1-\zeta}}\leq C(1+t)^{-\frac{3}{2}(1-\frac{1}{p})},
\end{align*}
where $\zeta= ({6-p})/ {2p} \in [0,1]$.

From Lemma \ref{L2.1}, we get
\begin{align}\label{G4.34}
\|(\rho,u,\theta,\eta)\|_{L^{\infty}}\leq\,& C\|\nabla (\rho,u,\theta,\eta)\|_{L^{2}}^{\frac{1}{2}}\|\Delta (\rho,u,\theta,\eta)\|_{L^{2}}^{\frac{1}{2}} 
\end{align}
For $p\in [6,\infty]$, using  Lemma \ref{L2.4} again, it follows from \eqref{G4.30} and \eqref{G4.34} that
\begin{align*}
&\|(\rho,u,\theta,\eta)\|_{L^p}\leq C\|(\rho,u,\theta,\eta)\|_{L^6}^{\zeta^{\prime}}\|(\rho,u,\theta,\eta)\|_{L^{\infty}}^{{1-\zeta^{\prime}}}\leq C(1+t)^{-\frac{3}{2}(1-\frac{1}{p})},
\end{align*}
where $\zeta^{\prime}= {6}/{p} \in [0,1]$.
For $p\in [2,6]$, by using Lemma \ref{L2.1}, it follows from \eqref{G4.30} and
\eqref{G4.32} that
\begin{align*}
&\|\nabla(\rho,u,\theta,\eta)\|_{L^p}\leq C\|\nabla(\rho,u,\theta,\eta)\|_{L^2}^{\eta^{\prime}}\|\nabla(\rho,u,\theta,\eta)\|_{L^6}^{{1-\eta^{\prime}}}\leq C(1+t)^{-\frac{3}{2}(\frac{4}{3}-\frac{1}{p})},
\end{align*}
where $\eta^{\prime}= ({6-p})/ {2p} \in [0,1]$.
Thus, we complete the proof of Theorem \ref{T1.2}.
\end{proof}

\medskip
\appendix  
\setcounter{table}{0}   
\setcounter{figure}{0}
 \setcounter{section}{1}
\setcounter{equation}{0}
\renewcommand{\thetable}{A\arabic{table}}
\renewcommand{\thefigure}{A\arabic{figure}}
\renewcommand{\thesection}{A\arabic{section}}
\renewcommand{\theequation}{A\arabic{equation}}
\numberwithin{equation}{section}

\section*{Appendix. Low and  medium
frequency parts  of strong solutions to the  linearized system}
 
Consider the following linearized version of the problem \eqref{I-2}--\eqref{I--2}:
\begin{equation}\label{A-1}
\left\{
\begin{aligned}
& \partial_t \rho+{\rm div}u=0, \\
& \partial_t u+\nabla\rho+\nabla\theta-\Delta u-2\nabla {\rm div}u+\nabla\eta=0, \\
& \partial_t \theta+{{\rm div}u}-\Delta \theta+4\theta-\eta=0,\\
& \partial_t \eta-\Delta\eta+\eta-4\theta=0,
\end{aligned}\right.
\end{equation}
with the initial data
\begin{align}\label{A--1}
(\rho,u,\theta,\eta(x,t)|_{t=0}&=( \rho_{0}(x),u_{0}(x),\theta_{0}(x),\eta_0(x))\nonumber\\
&=(\varrho_{0}(x)-1,u_{0}(x),\Theta_{0}(x)-1, n_0(x)-1 ).
\end{align}

In this appendix, we establish the time-decay estimates of the low-medium
frequency parts of strong solutions to the   problem  \eqref{A-1}--\eqref{A--1}.
We will divide \eqref{A-1} into two parts \eqref{A-2} and \eqref{A-3} to
get the decay estimates of $(\rho,u,\theta,\eta)$. Our method is in the same spirit of  
\cite{DD-JEE-2014}.

\subsection{Analysis of low-frequency part}
Motivated by \cite{Dan-00}, we apply the so-called ``Hodge decomposition"
to the velocity field. Setting $d=\Lambda^{-1}{\rm div}u$ and $\mathcal{P}u=\Lambda^{-1}{\rm curl}u$ and using the identity $\Delta=\nabla{\rm div}-{\rm curl}{\rm curl}$, we have the following identity:
\begin{align}\label{G5.3}
{u}=-\Lambda^{-1}\nabla d-\Lambda^{-1}{\rm curl}(\mathcal{P}{u}),   
\end{align}
where ${\rm curl}_{ij}=\partial_{x_j}u^i-\partial_{x_i}u^j$ and $\mathcal{P}$ is
a projection operator.
It is easy to see that $(\rho,u,\theta,\eta)$ and $\mathcal{P}u$ satisfy
\begin{equation}\label{A-2}
\left\{
\begin{aligned}
& \partial_t \rho+\Lambda d=0, \\
& \partial_t d-\Lambda\rho-\Lambda\theta-3\Delta d-\Lambda\eta=0, \\
& \partial_t \theta+\Lambda d-\Delta \theta+4\theta-\eta=0,\\
& \partial_t \eta-\Delta\eta+\eta-4\theta=0,
\end{aligned}\right.
\end{equation}
and
\begin{align}\label{A-3}
\partial_t(\mathcal{P}u)-\Delta(\mathcal{P}u)=0.    
\end{align}

By performing the Fourier transform on \eqref{A-2} and \eqref{A-3}, we have
\begin{equation}\label{G5.6}
\left\{
\begin{aligned}
& \partial_t \widehat{\rho}+|\xi| \widehat{d}=0, \\
& \partial_t \widehat{d}-|\xi|\widehat{\rho}-|\xi|\widehat{\theta}+3|\xi|^2\widehat{d}-|\xi|\widehat{\eta}=0, \\
& \partial_t \widehat{\theta}+|\xi|\widehat{d}+|\xi|^2\widehat{\theta} +4\widehat{\theta}-\widehat{\eta}=0,\\
& \partial_t \widehat{\eta}+|\xi|^2\widehat{\eta}+\widehat{\eta}-4\widehat{\theta}=0,
\end{aligned}\right.
\end{equation}
and
\begin{align}\label{G5.7}
\partial_t(\widehat{\mathcal{P}u})+|\xi|^2(\widehat{\mathcal{P}u})=0.    
\end{align}
Similar to \eqref{G3.52}, we consider the following variable change:
\begin{align*} \left(
\begin{matrix}
\widehat{\rho}  \\
\widehat{u}   \\
\widehat{g}  \\
\widehat{f}  \\
\end{matrix}\right):=\mathbb{Q}\left(
\begin{matrix}
\widehat{\rho}  \\
\widehat{u}   \\
\widehat{\theta}  \\
\widehat{\eta} \\
\end{matrix}\right),
\end{align*}
where
\begin{align*}
\mathbb{Q}=\left(
\begin{matrix}
1 & 0 & 0 & 0   \\
0 & 1 & 0 & 0    \\
0 & 0 & 4 & -1   \\
0 & 0 & 1 & 1  \\
\end{matrix}\right).
\end{align*}
Hence, \eqref{G5.6} can be expressed as
\begin{equation}\label{G5.10}
\left\{
\begin{aligned}
& \partial_t \widehat{\rho}+|\xi| \widehat{d}=0, \\
& \partial_t \widehat{d}-|\xi|\widehat{\rho}+3|\xi|^2\widehat{d}-|\xi|\widehat{f}=0, \\
& \partial_t \widehat{g}+4|\xi|\widehat{d}+|\xi|^2\widehat{g} +5\widehat{g}=0,\\
& \partial_t \widehat{f}+|\xi|\widehat{d}+|\xi|^2\widehat{f}=0.
\end{aligned}\right.
\end{equation}
It follows from \eqref{G5.10} that
\begin{align}\label{G5.11}
\frac{1}{2}\frac{\mathrm{d}}{\mathrm{d}t}\Big(|\widehat{\rho}|^2+|\widehat{d}|^2   +|\widehat{f}|^2+|\widehat{g}|^2    \Big)+3|\xi|^2|\widehat{d}|^2+|\xi|^2|\widehat{f}|^2+|\xi|^2|\widehat{g}|^2+5|\widehat{g}|^2 =-4|\xi|\mathbf{Re}(\widehat{d}\,|{\widehat{\overline{g}}}).     
\end{align}
By Young's inequality, we have
\begin{align}\label{G5.12}
-4|\xi|\mathbf{Re}(\widehat{d}{\widehat{\overline{g}}})\leq |\xi|^2|\widehat{d}|^2+4|\widehat{g}|^2.    
\end{align}
Putting \eqref{G5.12} into \eqref{G5.11}, we obtain 
\begin{align}\label{G5.13}
\frac{1}{2}\frac{\mathrm{d}}{\mathrm{d}t}\Big(|\widehat{\rho}|^2+|\widehat{d}|^2   +|\widehat{f}|^2+|\widehat{g}|^2    \Big)+2|\xi|^2|\widehat{d}|^2+|\xi|^2|\widehat{f}|^2+|\xi|^2|\widehat{g}|^2+|\widehat{g}|^2 \leq 0.   
\end{align}
Next, we give the estimate of $|\xi|^2|\widehat{\rho}|^2$.
From \eqref{G5.10}$_1$--\eqref{G5.10}$_2$, we directly get
\begin{align}\label{G5.14}
\frac{\mathrm{d}}{\mathrm{d}t}\Big(-|\xi|\mathbf{Re}(\widehat{\rho}{\widehat{\overline{d}}})  \Big)+|\xi|^2|\widehat{\rho}|^2=&\,-|\xi|^2\mathbf{Re}(\widehat{\rho}{\widehat{\overline{f}}})+3|\xi|^3\mathbf{Re}(\widehat{\rho}{\widehat{\overline{d}}}) +|\xi|^2|\widehat{d}|^2 \nonumber\\
\leq &\,\frac{1}{2}|\xi|^2|\widehat{\rho}|^2+\frac{1}{2}|\xi|^2|\widehat{f}|^2
+\frac{1}{4}|\xi|^2|\widehat{\rho}|^2+9|\xi|^4|\widehat{d}|^2+|\xi|^2|\widehat{d}|^2.
\end{align}
Then, plugging \eqref{G5.14} into \eqref{G5.13}, it holds
\begin{align}\label{G5.15}
\frac{1}{2}\frac{\mathrm{d}}{\mathrm{d}t}\mathbf{L}(t,\xi)+\frac{1}{4}|\xi|^2|\widehat{\rho}|^2+(2|\xi|^2-9|\xi|^4)|\widehat{d}|^2+\frac{1}{2}|\xi|^2|\widehat{f}|^2+|\xi|^2|\widehat{g}|^2+|\widehat{g}|^2 \leq 0,   
\end{align}
where the Lyapunov functional $\mathbf{L}(t,\xi)$ is defined as 
\begin{align*}
\mathbf{L}(t,\xi)=\big(|\widehat{\rho}|^2+|\widehat{d}|^2   +|\widehat{f}|^2+|\widehat{g}|^2-|\xi|\mathbf{Re}(\widehat{\rho}{\widehat{\overline{d}}})    \big).   
\end{align*}
For $|\xi|<\frac{1}{4}$, it holds that
\begin{align}\label{G5.17}
\frac{1}{2} \mathbf{L}(t,\xi)\leq \big(|\widehat{\rho}|^2+|\widehat{d}|^2   +|\widehat{f}|^2+|\widehat{g}|^2    \big)\leq 2\mathbf{L}(t,\xi).   
\end{align}
Combining \eqref{G5.15} with \eqref{G5.17}, there exists a positive constant $C_{11}$ independent of $|\xi|$ such that
\begin{align*}
\frac{\mathrm{d}}{\mathrm{d}t}\mathbf{L}(t,\xi)+C_{11}|\xi|^2\mathbf{L}(t,\xi) \leq 0.      
\end{align*}
This together with \eqref{G5.17} and Gronwall's inequality yields
\begin{align*}
|(\widehat{\rho},\widehat{d},\widehat{f},\widehat{g})(t,\xi)|^2\leq C{e}^{-C_{11}|\xi|^2t}|(\widehat{\rho},\widehat{d},\widehat{f},\widehat{g})(0,\xi)|^2.
\end{align*}
Noticing  the relation between $(\widehat{\rho},\widehat{d},\widehat{f},\widehat{g})$
and $(\widehat{\rho},\widehat{d},\widehat{\theta},\widehat{\eta})$,
we arrive at
\begin{align}\label{G5.20}
|(\widehat{\overline{\rho}},\widehat{\overline{d}},\widehat{\overline{\theta}},\widehat{\overline{\eta}})(t,\xi)|^{2}\leq C{e}^{-C_{11}|\xi|^{2}t}|(\widehat{\rho},\widehat{d},\widehat{\theta},\widehat{\eta})(0,\xi)|^{2},   
\end{align}
for $|\xi|<\frac{1}{4}$.

For the estimate of $\widehat{\mathcal{P} \overline{u}}$, it follows from
\eqref{G5.7} that
\begin{align*}
\frac{1}{2}\frac{\mathrm{d}}{\mathrm{d}t}|\widehat{\mathcal{P}\overline{u}}(t)|^{2}+|\xi|^2|\widehat{\mathcal{P}\overline{u}}(t)|^{2}=0,    
\end{align*}
for any $|\xi|\geq 0$,
which implies that 
\begin{align}\label{G5.22}
|\widehat{\mathcal{P}\overline{u}}(t)|^{2}\leq {e}^{-|\xi|^2t}|\widehat{\mathcal{P}u}(0)|^{2},   
\end{align}
for any $|\xi|\geq 0$.
\subsection{Analysis of medium-frequency part} In this subsection,
 we give the estimates of strong solutions for the medium-frequency part.

Now, we rewrite \eqref{G5.6} as  
\begin{align*}
\partial_t \widehat{\mathbb{V}}=-\mathbb{J}(\xi)\widehat{\mathbb{V}},   
\end{align*}
where $\widehat{\mathbb{V}}:=(\widehat{\overline{\rho}},\widehat{\overline{d}},\widehat{\overline{\theta}},\widehat{\overline{\eta}})^{T}$ and
\begin{align*}
\mathbb{J}(\xi)=\left(
\begin{matrix}
0 & |\xi| & 0 & 0   \\
-|\xi| & 3|\xi|^2 & -|\xi| & -|\xi|    \\
0 & |\xi| & |\xi|^2+4 & -1   \\
0 & 0 & -4 & |\xi|^2+1  \\
\end{matrix}\right).
\end{align*}
Then, through Routh-Hurwitz theorem (see \cite{LB-1972}, p.459), we
conclude that all eigenvalues of matrix $\mathbb{J}(\xi)$ have
the positive real part if and only if the determinants $A_1$,
$A_2$, $A_3$, $A_4$ below are positive. By a direct calculation, we get the characteristic 
polynomial of $\mathbb{J}(\xi)$:
\begin{align*}
{\rm det}|\lambda\mathbb{I}-\mathbb{J}(\xi)|=&\lambda^4-(5|\xi|^2+5)\lambda^3+(7|\xi|^4+22|\xi|^2)\lambda^2\nonumber\\
&\,-(3|\xi|^6+18|\xi|^4+10|\xi|^2)\lambda+(|\xi|^6+5|\xi|^4).
\end{align*}

Similar to the method in Section 3.3 in \cite{DD-JEE-2014}, 
we denote $$h(\lambda):=a_0\lambda^4-a_1\lambda^3+a_2\lambda^2-a_3\lambda+a_4,$$ 
where
$$a_0=1 ,  a_1=5|\xi|^2+5 ,  a_2=7|\xi|^4+22|\xi|^2 , 
 a_3=3|\xi|^6+18|\xi|^4+10|\xi|^2 ,  a_4=|\xi|^6+5|\xi|^4.$$
Denote
\begin{align*}
A_1:= a_1,\ \
A_2:=\begin{vmatrix}
a_1 & a_0  \\
a_3  & a_2\\
\end{vmatrix},  \ \
A_3:=\begin{vmatrix}
a_1 & a_0 & 0 \\
a_3  & a_2 & a_1\\
0  & a_4  &a_3\\
\end{vmatrix}, \ \ 
A_4:=\begin{vmatrix}
a_1 & a_0 & 0 &0\\
a_3  & a_2 & a_1&0\\
0 & a_4  &a_3&0\\
0  & 0 &0  &a_4\\
\end{vmatrix}. 
\end{align*}
By tedious calculation, we directly have
\begin{align*}
A_1&=\,5|\xi|^2+5>0,\\
A_2&=\,32|\xi|^6+127|\xi|^4+100|\xi|^2>0,\\
A_3&=\,96|\xi|^{12}+932|\xi|^{10}+2731|\xi|^8+2795|\xi|^6+875|\xi|^4>0.
\end{align*}
In addition, from the fact that sgn$A_3$=sgn$A_4$, we derive that  $A_4>0$.

According to the analysis of medium-frequency part of solutions,
we directly have the following proposition.

\begin{prop}
\label{P5.1}
there exists constants $\kappa>0$, such that    
\begin{align}\label{G5.26}
|e^{-t\mathbb{J}(\xi)}|\leq Ce^{-\kappa t},    
\end{align}
for $r<|\xi|< R$ and $t\geq 0$. Here, $\kappa$ and $C$ are independent of $|\xi|$. 
\end{prop}

For $r\leq |\xi| \leq R$, it follows from \eqref{G5.26} that
\begin{align*}
|(\widehat{\overline{\rho}},\widehat{\overline{d}},\widehat{\overline{\theta}},\widehat{\overline{\eta}})(t)|\leq C{e}^{-\kappa t}|(\widehat{\rho},\widehat{d},\widehat{\theta},\widehat{\eta})(0)|.  
\end{align*}


\bigskip 
{\bf Acknowledgements:} Li and   Ni  are supported by NSFC (Grant Nos. 12331007, 12071212).  
And Li is also supported by the ``333 Project" of Jiangsu Province.
Jiang  is supported by NSF of Jiangsu Province 
(Grant No. BK20191296). 

\bibliographystyle{plain}

\end{document}